\documentclass[reqno]{amsart}
\usepackage[centertags]{amsmath}
\usepackage{amsfonts}
\usepackage{amssymb}
\usepackage{amsthm}
\usepackage{newlfont}
\usepackage[top=1in, bottom=1in, left=1in, right=1in]{geometry}
\usepackage{mathrsfs}
\usepackage{hyperref}

\newtheorem{theorem}{Theorem}[section]
\newtheorem{proposition}[theorem]{Proposition}
\newtheorem{lemma}[theorem]{Lemma}
\newtheorem{corollary}[theorem]{Corollary}

\theoremstyle{definition}
\newtheorem{definition}{Definition}
\newtheorem*{notation}{Notation}

\theoremstyle{remark}
\newtheorem*{remark}{Remark}
\newtheorem*{example}{Example}

\theoremstyle{plain}
\newtheorem{ax}{Axiom}
\newtheorem*{ax*}{Axiom}
\newtheorem*{conj}{Conjecture}

\newcommand{\N}{\mathbb{N}}
\newcommand{\Z}{\mathbb{Z}}
\newcommand{\C}{\mathbb{C}}

\DeclareMathOperator{\Fix}{Fix}
\DeclareMathOperator{\id}{id}

\newcommand{\q}[1]{\mathscr{Q}_{#1}}
\newcommand{\Q}{\mathscr{Q}}
\newcommand{\jac}{\mathcal{J}}
\newcommand{\tensor}{\otimes}

\newcommand{\FJRW}{\mathscr{H}}
\newcommand{\h}[2]{\FJRW_{#1, #2}}

\DeclareMathOperator{\hess}{hess}

\newcommand{\lb}{l}

\newcommand{\1}{\mathbf{1}}

\newcommand{\set}[1]{\left\{ #1 \right\}}

\newcommand{\br}[1]{\left\langle #1 \right\rangle}

\newcommand{\ket}[1]{\left|\,#1\right>}
\newcommand{\bs}[1]{\boldsymbol{#1}}

\newcommand{\parlengths}{\setlength{\parindent}{0pt}}
\begin{document}

\date{\today}
\title{FJRW-rings and Landau-Ginzburg Mirror Symmetry}

\author{Marc Krawitz}
\thanks{Partially Supported by the National Research Foundation of South Africa}

\address{Department of Mathematics, University of Michigan, Ann Arbor, MI, USA}
\email{mkrawitz@umich.edu}

\begin{abstract}
{In this article, we study the Berglund--H\"ubsch transpose  construction $W^T$ for invertible quasihomogeneous potential $W$. We introduce the
dual group $G^T$ and establish the state space isomorphism between the Fan--Jarvis--Ruan--Witten $A$-model of $W/G$ and the orbifold Milnor ring $B$-model of $W^T/G^T$. Furthermore, we prove a mirror symmetry theorem at the level of Frobenius algebra structure
 for $G^\text{max}$. Then, we interpret Arnol'd strange duality of exceptional singularities $W$ as mirror symmetry between $W/\br{J}$ and its strange dual $W^{SD}$.}
 \end{abstract}
\maketitle\parlengths
\tableofcontents
\section{Introduction}\label{sec:Introduction}

During the last twenty years, mirror symmetry has been a driving force for some of the developments in geometry and physics.  In this article, we add to this development a version of mirror symmetry purely in the Landau-Ginzburg / singularity setting, i.e. we produce a mirror LG theory to a given LG theory.  This version of mirror symmetry is inspired by an early proposal of Berglund--H\"ubsch \cite{BH} for invertible singularities and the recent development of quantum singularity theory / LG topological string of Fan--Jarvis--Ruan--Witten \cite{FJR1}.  Compared to the other forms of mirror symmetry such as Calabi-Yau via Calabi-Yau and toric via LG, our version has the benefit of not having any poorly behaved exceptional cases.  These exceptional cases hindered the formulation of mathematical conjectures in the other forms of mirror symmetry.  This makes the present (LG via LG) model of mirror symmetry attractive for future study.  Historically, this version has been used in physics to verify geometric mirror symmetry (CY via CY) through the conjectured LG / CY correspondence.  LG via LG is certainly more general, since the LG-orbifold theories under consideration do not have to correspond to Calabi-Yau manifolds.  Even if they do correspond to Calabi-Yau manifolds (orbifolds), they are not necessarily the Gorenstein orbifolds where a mathematical proof was established by Batyrev \cite{Batyrev}.  The author has been informed that this generality, combined with a proof of LG / CY correspondence, has been exploited by Chiodo-Ruan \cite{ChiodoRuanPrivate} to generalize Batyrev's theorem in Calabi-Yau hypersurface of Gorenstein weight projective spaces.

LG via LG mirror symmetry is not a new idea.  As indicated, it was an important physical tool to verify Calabi-Yau mirror symmetry in the early investigations of this phenomenon.  Throughout the literature, a striking construction was Berglund--H\"ubsch's trasposed potential.  Let me briefly review this construction, which restricts consideration to so-called \emph{invertible} singularities.  A quasi-homogeneous polynomial $W=\sum_{i=1}^s c_i\prod_{j=1}^N z_j^{a_{ij}}$ is invertible if $s=N$ (\emph{i.e.} if the number of monomials equals the number of variables).  Then the matrix $A_W=(a_{ij})$ of exponents is square, and the matrix $A_W^T$ defines another quasi-homogeneous invertible potential -- the \emph{Berglund--H\"ubsch dual} -- which we denote by $W^T=\sum_{i=1}^s c_i\prod_{j=1}^N z_j^{a_{ji}}$.  Berglund--H\"ubsch proposed almost twenty years ago \cite{BH} that $(W,W^T)$ forms a mirror pair.  It was known that one must consider orbifold LG-models $(W/G, W^T/G^T)$ for this proposition to have any chance to be correct.  In the literature, the construction of the dual group $G^T$ is known in many cases (e.g. for the Fermat Quintic), and we present a general construction in Section \ref{sec:Duality of Groups}.  We should emphasize that the subject of LG via LG mirror symmetry was never fully developed in physics because (i) a construction of the $A$-model was absent, and (ii) although the orbifold $B$-model state space was given by Intriligator--Vafa \cite{IV}, the ring structure was still lacking.  The first problem was solved recently by Fan--Jarvis--Ruan--Witten \cite{FJR1}-\cite{FJR3} with the establishment of quantum singularity theory / LG-topological string.  As for the second problem, Kaufmann wrote down the multiplication in many cases and proposed a general recipe \cite{Ka1}-\cite{Ka3}.  Guided by his recipe, we wrote down a multiplication for non-degenerate invertible singularities $W$ and $G\subset SL$.  Our definition of multiplication has an important restriction not present in Kaufmann's recipe, namely that the $B$-model orbifold group should be a subgroup of $SL_N\C$.  This is dual to the fact that in Fan--Jarvis--Ruan--Witten's construction, every admissible $A$-model orbifold group must contain the exponential grading operator $J$.

With both the $A$-model and the $B$-model established, we can study the LG via LG mirror symmetry.  Our main theorems are:

\begin{theorem}\label{thm:state space isomorphism intro}
	Let $W$ be a non-degenerate invertible potential and $G$ a group of diagonal symmetries of $W$.
  There is a bi-graded isomorphism of vector spaces
  \[\h WG \cong \Q_{W^T, G^T},\]
  where $\h WG$ is the FJRW $A$-model of $(W,G)$ and $\Q_{W^T, G^T}$ is the orbifold $B$-model of $(W^T, G^T)$.
\end{theorem}

\begin{theorem}\label{thm:frobenius algebra isomorphism intro}
	Let $W$ be a non-degenerate invertible potential and $G^\text{max}$ its maximal group of diagonal symmetries.  There is an isomorphism of Frobenius algebras
  \[\h {W}{G^\text{max}} \cong \Q_{W^T},\]
  where $\Q_{W^T}$ is the unorbifolded $B$-model of $W^T$.
\end{theorem}

Invertible singularities include, for example, Arnol'd's list of simple, unimodal and bimodal singularities \cite{AGV}.  Theorem \ref{thm:frobenius algebra isomorphism intro} has already been proven for the simple and parabolic singularities \cite{FJR1} and the unimodal and bimodal singularities \cite{KP+}.  The 14-families of exceptional (unimodal) singularities exhibit the famous \emph{Arnol'd strange duality}.  There have been attempts to explain strange duality by relating exceptional singularities to K3-surfaces.  It is possibly more natural to consider it from the LG via LG mirror symmetry perspective.  For example, we apply Theorem \ref{thm:frobenius algebra isomorphism intro} to show that strange duality indeed agrees with LG via LG mirror symmetry.

\begin{corollary}\label{cor:strange duality intro}
Let $W$ be one of Arnol'd's 14 exceptional singularities with strange dual $W^\text{SD}$, and $J$ its exponential grading operator.  Then
\[\h{W}{\br{J}}\cong \Q_{W^\text{SD}}.\]
i.e. the LG $A$-model for $W$ orbifolded by $J$ is isomorphic (as a Frobenius algebra) to the unorbifolded LG $B$-model of $W^\text{SD}$.
\end{corollary}

   We should mention that Theorem \ref{thm:state space isomorphism intro} specializes in the case of $G=G^\text{max}$ to the main result of Kreuzer in \cite{Kreuzer}.  That work considers only a single grading, and appeals to physically motivated `twist selection rules' to argue that the mirror map is degree-preserving.
   We clarify the physical picture, and establish our theorems in the most general context (bi-grading, dual group, Frobenius algebra structure) in order to set the stage for the future applications of LG-mirror symmetry. As we mentioned previously, one important application is the CY via CY mirror symmetry of CY-hypersurfaces of non-Goreinstein weighted projective space which is exploited by Chiodo-Ruan.

   Another important application is the integrable hierarchies problem.  Recall that there is a semi-simple Frobenius manifold theory for the unorbifolded $B$-model of $W^T$ due to Saito and the high-genus theory by Givental. Theorem \ref{thm:frobenius algebra isomorphism intro} naturally suggests the following conjecture
    \begin{conj}\label{conj:Frobenius Manifolds}
    Let $W$ be a non-degenerate invertible potential and $G^\text{max}$ be its maximal group of diagonal symmetries. Then the full FJRW-theory of $W/G$ is isomorphic to Saito--Givental theory of $W^T$.
    \end{conj}
    In many cases,  the Saito--Givental theory of $W^T$ is expected to satisfy certain integrable hierarchies. The study of these examples leads to a generalization of Witten's famous ADE integrable hierarchies conjecture solved by Fan--Jarvis--Ruan \cite{FJR1}. We refer the interested readers to
    \cite{RuanSurvey} for the details.

The paper is organized as follows.

\subsection{Organization of paper}\label{sec:Organization of paper}
We present some basic notions regarding invertible potentials in Section \ref{sec:Kreuzer-Skarke Classification}, including Kreuzer--Skarke's classification of invertible potentials.

In Section \ref{sec:The A and B models} we review the construction of the FJRW $A$-model Frobenius algebra, as well as the orbifold $B$-model state space of Intriligator--Vafa.  We introduce a multiplication on the orbifold $B$-model and show that this multiplication respects a suitably shifted version of the bi-grading of Intriligator--Vafa.

In Section \ref{sec:Mirror Symmetry} we prove an LG-via-LG mirror symmetry result, after introducing a suitable notion of duality between the symmetry groups of Berglund--H\"ubsch dual potentials.

In Section \ref{sec:Mirror Symmetry for Frobenius Algebras}, we show that there is an isomorphism of Frobenius algebras between the maximally orbifolded $A$-model of a singularity and the unorbifolded $B$-model of the Berglund--H\"ubsch dual. Finally, we demonstrate a relation between Arnol'd's strange duality and the LG-via-LG mirror symmetry discussed in this paper.
\subsection{Acknowledgements}\label{sec:Acknowledgements}
Thanks are due to several people, without whom this work would not be appearing in its present form.  I benefited greatly from an invitation of Tyler Jarvis to visit Brigham Young University and present an early incarnation of this work.  While there, I met several students working on similar material, with whom I collaborated to produce \cite{KP+}.  They have kindly permitted me to present in Section \ref{sec:A-model} a very slightly amended version of the `Review of Construction' section of that paper.  I enjoyed fruitful discussions about this work with Huijin Fan, Takashi Kimura, and Ralph Kaufmann, and am grateful for the extended contact I have had with Alessandro Chiodo, whose combination of enthusiasm and expertise I can only aspire to.

I owe an inestimable debt to Yongbin Ruan.  From help navigating the physics literature to imparting tips and tricks for executing messy computations, his advice and support have been unwavering.
\subsection{Preliminaries on Invertible Potantials}\label{sec:Kreuzer-Skarke Classification}
Consider an invertible non-degenerate quasihomogeneous potential
  \[W = \sum_{i=1}^N c_i \prod_{j=1}^N X_j^{a_{ij}}.\]

The exponent matrix $A = (a_{ij})$ encodes the singularity, modulo the coefficients $c_i$ of the monomials.

The charges (or weights) $q_i$ are determined by the condition that
\begin{equation}
A
\begin{bmatrix}
q_1\\
\vdots\\
q_N
\end{bmatrix}
=
\begin{bmatrix}
1\\
\vdots\\
1
\end{bmatrix}.
\label{eq:q relations matrix form}
\end{equation}
\begin{remark}
Since $A$ is invertible, the $c_i$ may be absorbed by rescaling the variables.  In what follows, we will take $c_i=1$ without loss of generality.
\end{remark}
If the matrix $A_W$ is square, its transpose $A_W^T$ will also correspond to a quasi-homogeneous polynomial, which we denote by $W^T$.

Non-degeneracy of $W$ requires the charges to be uniquely determined, so $\det A \neq 0$.

We write
\begin{equation} \label{eq:A inverse columns = rho}
A^{-1} =
\Biggl(\begin{array}{c|c|c|c}
\rho_k & \rho_2 & \cdots & \rho_N
\end{array}\Biggr),
\phantom{XX} \text{ with column vectors }
\phantom{XX} \rho_k=
\begin{bmatrix}
\varphi^{(k)}_1\\
\vdots\\
\varphi^{(k)}_N
\end{bmatrix}
\end{equation}
Then each $\rho_k$ defines a symmetry of $W$ via
\[\rho_k X_j = \exp(2\pi i \varphi^{(k)}_j) X_j.\]  We will abuse notation and use the same symbol to denote the symmetry and the column vector.
\begin{remark}
Suppose $g\subset (\C^*)^N$ is a diagonal symmetry of $W$, with $g X_k = \exp(2\pi i g_k)X_k$.  $g$ preserving $W$ is equivalent to
\[A\begin{bmatrix}
g_1\\
\vdots\\
g_N
\end{bmatrix} \in \Z^N,\]
So the phase vector $(g_1,\dotsc,g_N)^T$ is a linear combination of the columns of $A^{-1}$.  This implies that the $\rho_k$ generate the group $G_\text{max}$ of diagonal symmetries of $W$. For any $g\in G^\text{max}$, we can write $g=\prod_{i=1}^N \rho_i^{\alpha_i}$.
\end{remark}
\begin{remark}
The group $G^\text{max}$ is non-trivial, as it contains the \emph{exponential grading operator} $J$, which acts on $X_k$ with phase $q_k$.

Multiplying Equation \eqref{eq:q relations matrix form} by $A^{-1}$, we see that $J$ is given by
\begin{equation}\label{eq: J = product of rho_i}
  J = \prod_{i=1}^N \rho_i
\phantom{XX}\text{\textit {i.e.}}\phantom{XX}
|J|=\begin{bmatrix}
1\\
\vdots\\
1
\end{bmatrix}
\end{equation}
\end{remark}
In \cite{KS}, Kreuzer and Skarke prove that an invertible potential is non-degenerate if and only if it can be written as a sum of (decoupled) invertible potentials of one of the following three types, which we will refer to as \emph{atomic types}:
\begin{itemize}
\item[] \[W_\text{Fermat} = X^a.\]
\item[] \[W_\text{loop}= X_1^{a_1}X_2+X_2^{a_2}X_3+\dotsb +X_{N-1}^{a_{N-1}}X_N+X_N^{a_N}X_1.\]
\item[] \[W_\text{chain}= X_1^{a_1}X_2+X_2^{a_2}X_3+\dotsb +X_{N-1}^{a_{N-1}}X_N+X_N^{a_N}.\]
\end{itemize}
Although this classification allows for terms $X_kX_k+1$ (i.e. $a_k=1$), we will only consider the case $a_i\geq 2$ so that the charges satisfy $q_i\leq \tfrac{1}{2}$, as this condition is necessary for the construction of the FJRW $A$-model.
\begin{remark}
The proof of Theorem \ref{thm:Mirror Theorem for G_max} Chain potentials is valid only if $a_N > 2$, so that all charges are strictly less than $\tfrac{1}{2}$.
\end{remark}
It is clear that the transpose construction $W^T$ preserves the above types.
Our arguments will rely heavily on an understanding of these `atomic' potentials and their symmetry groups, and we recall without proof some elementary facts from \cite{Kreuzer} below.  Because the Fermat potential is particularly straightforward, our discussion focuses on Loops and Chains.
\begin{notation}
We use the Dirac delta: \[\delta_{\alpha,\beta}:=\begin{cases}1& \text{ if } \alpha=\beta,\\
0&\text{ else.}\end{cases}\]
Also
\[\delta_{i}^\text{even}:=\begin{cases}1& \text{ if $i$ is even,}\\
0&\text{ else,}\end{cases}\]
with $\delta_i^\text{odd}$ defined similarly.
\end{notation}
We now recall without proof some facts from \cite{Kreuzer} which will be useful in the sequel.

The following lemma facilitates computation of the phase of a given symmetry on a variable $X_j$.
\begin{lemma}\label{lem:Atomic potentials - phase equals algebraic sum}
  Let $W\in\C[X_1,\dotsc,X_N]$ be a non-degenerate invertible potential of atomic type, with exponent matrix $A_W$ and generators of $G^\text{max}$ given by $\rho_1,\dotsc,\rho_N$ corresponding to the columns of $A_W^{-1}$.  Let $g=\left(\prod_{i=1}^N \rho_i^{\alpha_i}\right)$, with $0 \leq \alpha_i < a_i$.
  For $j\in\{1,\dotsc,N\}$ with $X_j$ not fixed by $gJ$,
  \[\Theta_j^{gJ}= \sum_{i=1}^N (\alpha_i+1)\varphi_j^{(i)}.\]
  \emph{i.e} The phase of $gJ$ on $X_j$ is given by the algebraic sum of the phases of the $\rho_i$ on $X_j$, without the need to reduce this sum modulo $1$.

If $X_j$ is fixed by $gJ$, $\Theta_j^{gJ}=0$ although the algebraic sum of phases may equal either 0 or 1 (and these cases can be explicitly identified).
\end{lemma}
The following lemma gives explicit generators over $\C$ for the Milnor ring of a loop or chain potential.
\begin{lemma}\label{lem:Atomic potentials -- Milnor Ring}
$\,$

  \begin{itemize}
  \item The Milnor ring $\Q_{W_\text{loop}}$ for a loop potential is generated over $\C$ by $\{\prod_{i=1}^N X_i^{\alpha_i}\,|\,0\leq \alpha_i < a_i\}$, and has dimension $\prod_{i=1}^N a_i$.
  \item The Milnor ring $\Q_{W_\text{chain}}$ for a chain potential is generated over $\C$ by $\{\prod_{i=1}^N X_i^{\alpha_i}\,|\,0\leq \alpha_i < a_i\}$ subject to the condition that the largest set $\{1,\dotsc,s\}$ of consecutive indices for which $\alpha_i=\delta_{i}^\text{odd}(a_i-1)$ has an even number of elements (possibly zero).  Its dimension is $\sum_{i =\text{odd}}(a_i-1)\prod_{j=i+1}^N a_j$,
  where we interpret the empty product as equal to $1$.
  \end{itemize}
\end{lemma}
For $g\in G^\text{max}$, the next lemma identifies the $G^\text{max}$-invariants in $\Q_{\Fix(gJ)}$.
The lemma after next one identifies the symmetry groups of atomic invertible potentials.
\begin{lemma}\label{lem:Atomic potentials -- symmetries}$\,$\setlength{\parindent}{0pt}

\begin{itemize}
\item Let $W_\text{loop}$ be a loop potential.  Then $G_W^\text{max}$ has order $\prod_{i=1}^N a_i - (-1)^N$.

If $N$ is even, any symmetry $g$ of $W_\text{loop}$ may be written
$$g=\prod_{i=1}^N\rho_i^{\alpha_i} \phantom{XX}\text{ with }\phantom{XX} 0\leq\alpha_i < a_i.$$
This presentation is unique, except in the case of $J^{-1}$
$$J^{-1}=\prod_{i \text{ even}} \rho_i^{a_i-1} \phantom{XX}\text{ and }\phantom{XX} J^{-1}=\prod_{i \text{ odd}} \rho_i^{a_i-1}$$
If $N$ is odd, any symmetry $g\neq J^{-1}$ of $W_\text{loop}$ may be written uniquely as
$$g=\prod_{i=1}^N\rho_i^{\alpha_i} \phantom{XX}\text{ with }\phantom{XX} 0\leq\alpha_i < a_i.$$
  \item Let $W_\text{chain}$ be a chain potential.  Then $G_W^\text{max}$ has order $\prod_{i=1}^N a_i$, and any $g   \in G_W^\text{max}$ may be written uniquely as
  $$g=\prod_{i=1}^N\rho_i^{\alpha_i} \phantom{XX}\text{ with }\phantom{XX} 0\leq\alpha_i < a_i.$$
\end{itemize}
\end{lemma}
\begin{remark}
Lemmas \ref{lem:Atomic potentials -- Milnor Ring}, and \ref{lem:Atomic potentials -- symmetries} combine to show that for loop and chain potentials, image the map
\begin{align*}
\Q_{W_\text{chain}}^T&
\longrightarrow G_W\\
  \prod_{i=1}^N Y_i^{\alpha_i}dY_i&
  \longmapsto \left(\prod_{i=1}^N \rho_i^{\alpha_i}\right)J
\end{align*}
is the collection of group elements with \emph{even dimensional} fixed loci.  The map is injective for chains, and simply ramified over $J^{-1}$ for loops when $N$ is even.
\end{remark}
\begin{lemma}\label{lem:Atomic potentials -- Gmax invariants in twisted sector}
  $\,$

  \begin{itemize}
  \item For a loop potential $W_\text{loop}$, the only symmetry $gJ$ with non-trivial fixed locus is $gJ=\id$, which has fixed locus $\C^N$.  Generators of the $G^\text{max}$ invariants as a $\C$-vector space are given by
  \[\Q_{\Fix(gJ)}^{G^\text{max}} = \begin{cases}
  	\emptyset & \text{ if } gJ=\id, \text{ and $N$ is odd.}\\
    \left\{\prod_{i=1}^N X_i^{\delta_i^\text{even}(a_i-1)}dX_i,\, \prod_{i=1}^N X_i^{\delta_i^\text{odd}(a_i-1)}dX_i\right\} & \text{ if } gJ=\id, \text{ and $N$ is even.}\\
    1 & \text{ else.}
  \end{cases}\]
  \item For a chain potential, $W_\text{chain}$, if a symmetry $gJ$ fixes $X_t$, it must fix $\{X_t,\dotsc,X_N\}$.  $\Fix gJ = \{X_t,\dotsc,X_N\}$ implies $g=\prod_{i=1}^N \rho_i^{\alpha_i}$ has $\alpha_i = \delta_{N-i}^\text{even}(a_i-1)$ for $i\geq t$, and this relation does not hold for $i=t-1$

  The $G^\text{max}$-invariants in $\Q_{\Fix(gJ)}$ are generated by
  $$
  \Q_{\Fix(gJ)}^{G^\text{max}} = \begin{cases}
    \emptyset & \text{ if } \Fix(gJ)=\{X_t,\dotsc, X_N\} \text{ is odd-dimensional.}\\
    \prod_{i=t}^N X_i^{\delta_{i-t}^\text{even}(a_i-1)}dX_i,  & \text{ if } \Fix(gJ)=\{X_t,\dotsc, X_N\} \text{ is even-dimensional,}\\
    1 & \text{ if } \Fix(gJ)=\emptyset\\
  \end{cases}
  $$
  \end{itemize}
\end{lemma}
\section{The $A$ and $B$ models}\label{sec:The A and B models}
\subsection{FJRW $A$-model}\label{sec:A-model}
Let $W$ be a non-degenerate quasi-homogeneous polynomial in the variables $x_1, x_2, \dots, x_N$ with weights $q_1,q_2,\dots,q_N$ respectively. Non-degeneracy requires that these weights are uniquely determined by the condition that each monomial in $W$ has total weight $1$, and that $W$ has an isolated singularity at the origin.

The central charge of $W$ is defined to be
\[
\hat c:=\sum_{j=1}^N(1-2q_j).
\]
The Jacobian ideal $\jac$ is defined by \[\jac=\br{\frac{\partial W}{\partial x_1},\frac{\partial W}{\partial x_2},\dots,\frac{\partial W}{\partial x_N}}.
\]
The Milnor ring $\q{W}$ is given by
\[
\q{W} :=\C[x_1,x_2,\dots,x_N]/\jac
\]
together with the residue pairing. $\q{W}$ is a finite dimensional vector space over $\C$, with dimension
\[
\mu=\prod_{j=1}^N \left(\frac 1{q_j}-1\right).
\]
It is graded by weighted degree, and the elements of top degree form a one-dimensional subspace generated by $\text{hess}(W)=\det\left(\tfrac{\partial^2 W}{\partial x_i\partial x_j}\right)$.  One can check directly that the top degree is equal to $\hat c$.

For $f,g\in\q{W}$, the residue pairing $\br{f,g}$ may be defined by
\begin{equation}\label{eq:Milnor Ring Pairing}
fg = \frac{\br{f,g}}{\mu}\text{hess}(W) + \text{ lower order terms.}
\end{equation}
This pairing is non-degenerate, and endows the Milnor ring with the structure of a Frobenius algebra. For more details, see \cite{AGV}.

To define the FJRW ring, we require in addition to $W$ a choice of a group of diagonal symmetries of $W$. The choice of group heavily affects the resulting structure of the FJRW ring. The maximal group of diagonal symmetries is defined as
\[
G_W=\set{(\alpha_1,\alpha_2,\dots,\alpha_N)\subseteq (\C^*)^N \,|\,W(\alpha_1x_1,\alpha_2x_2,\dots,\alpha_Nx_N)=W(x_1,x_2,\dots,x_N)}
\]
Note that $G_W$ always contains the exponential grading element $J=(e^{2\pi iq_1},e^{2\pi iq_2},\dots,e^{2\pi iq_N})$. In general, the theory requires that the symmetry group be \emph{admissible} (see \cite{FJR1} section 2.3).

The Landau--Ginzburg Mirror Symmetry Conjecture states the following:
\begin{conj}[Landau--Ginzburg Mirror Symmetry Conjecture]
For a non-degenerate, quasi-homogeneous, singularity $W$ and diagonal symmetry group $G$, there is a dual singularity $W^T$ with dual symmetry group $G^T$ so that the FJRW-ring of $W/G$ is isomorphic to an orbifolded Milnor ring of $W^T/G^T$.\end{conj}
\begin{remark}
(i) We use the notation $W^T$ suggestively for the dual singularity, as we demonstrate in this paper that the Berglund--H\"ubsch transposed singularity is the appropriate dual in the context of LG via LG mirror symmetry for non-degenerate invertible potentials; (ii) A strengthening of the conjecture applying to Frobenius manifolds should be true. However, a definition of orbifolded Frobenius manifold of $W^T/G^T$ is currently lacking.
\end{remark}
We now outline the definition of $\h WG$ as a $\C$-vector space, after which we will define the pairing, grading, and multiplication that make $\h WG$ a Frobenius algebra.

In \cite{FJR1}, the state space $\h WG$ is defined in terms of Lefschetz thimbles:
\[\h WG = \bigoplus_{\gamma\in G}H^\text{mid}(\Fix\gamma, W_\gamma^\infty,\Q)^G.\]
For further details, see \cite{FJR1}.  For our purposes, it will be most convenient to give a presentation in terms of Milnor rings, but we should point out that the isomorphism between the two presentations is not canonical (\cite{Wall1}, \cite{Wall2}).

Let $G$ be an admissible group. For $h\in G$, let $\Fix h\subset\C^N$ be the fixed locus of $h$, and let $N_h$ be its dimension. Define
\[
\mathscr{H}_{h}:=\Omega^{N_h}(\C^{N_h})/\left(dW|_{\Fix h}\wedge \Omega^{N_h-1}\right)\cong\q{W|_{\Fix h}}\cdot\omega
\]
where $\omega=dx_{i_1}\wedge dx_{i_2}\wedge\dots\wedge dx_{i_{N_h}}$ is a volume form\footnote{Note the volume form encodes a determinant-twist on the natural $G$-action on $\Q_{W|_{\Fix h}}$.}.

$G$ acts on $\mathscr{H}_{h}$ via its action on the coordinates, and the state space of the FJRW-ring is the vector space of invariants under this action, i.e.
\[
\h WG := \left(\bigoplus_{h\in G}\mathscr{H}_h\right)^G.
\]
$\h WG$ is $\mathbb{Q}$-graded by the so-called $W$-degree, which depends only on the $G$-grading.  To define this grading, first note that each element $h\in G$ can be uniquely expressed as
\[
h=(e^{2\pi i\Theta_1^h},e^{2\pi i\Theta_2^h},\dots,e^{2\pi i\Theta_N^h})
\]
with $0\leq \Theta_i^h < 1$.

For $\alpha_h \in (\mathscr{H}_{h})^G$, the $W$-degree of $\alpha_h$ is defined by
\begin{equation}
\deg_W(\alpha_h):=\dim\Fix h+2\sum_{j=1}^N(\Theta_j^h-q_j).\label{degw}
\end{equation}
Since $\Fix h=\Fix h^{-1}$, we have $\mathscr{H}_h \cong \mathscr{H}_{h^{-1}}$, and the pairing on $\q{W|_{\Fix h}}$ induces a pairing
\[(\mathscr{H}_h)^G \otimes (\mathscr{H}_{h^{-1}})^G\to\C.\]
The pairing on $\h WG$ is  the direct sum of these pairings.  Fixing a basis for $\h WG$, we denote the pairing by a matrix $\eta_{\alpha,\beta}=\br{\alpha,\beta}$, with inverse $\eta^{\alpha,\beta}$.

For each pair of non-negative integers $g$ and $n$, with $2g-2+n>0$, the FJRW cohomological field theory produces classes $\Lambda_{g,n}^W(\alpha_1,\alpha_2,\dots,\alpha_N)\in H^*(\overline {\mathscr{M}}_{g,n})$  of complex codimension $D$ for each $n$-tuple $(\alpha_1,\alpha_2,\dots,\alpha_N)\in (\h WG)^N$. The codimension $D$ is given by
\[
D:=\hat c_W(g-1)+\frac 12\sum_{i=1}^N\deg_W(\alpha_i),
\]
and the $n$-point correlators are defined to be
\[
\br{\alpha_1,\dotsc,\alpha_N}_{g,n}:=\int_{\overline {\mathscr{M}}_{g,n}}\Lambda_{g,n}^W(\alpha_1,\dotsc,\alpha_N),
\]
so $\br{\alpha_1,\dotsc,\alpha_N}_{g,n}$ obviously vanishes unless the codimension of $\Lambda_{g,n}^W(\alpha_1,\dotsc,\alpha_N)$ is zero. The ring structure on $\mathscr{H}_{W,G}$ is determined by the genus-zero three-point correlators. In other words, if $r,s\in \h{W}{G}$, then
\begin{equation}\label{eq:FJRW multiplication definition}
r\star s:=\sum_{\alpha,\beta}\br{r,s,\alpha}_{0,3}\eta^{\alpha,\beta}\beta
\end{equation}
where the sum is taken over all choices of $\alpha$ and $\beta$ in a fixed basis of $\h WG$.

The classes $\Lambda_{g,n}^W(\alpha_1,\dotsc,\alpha_N)$ satisfy the following axioms which facilitate the computation of the three-point correlators $\br{\alpha_1,\alpha_2,\alpha_3}$.
\begin{ax}\label{ax:dimension} Dimension: If the codimension $D\notin \frac 12 \Z$, then $\Lambda_{g,n}^W(\alpha_1,\alpha_2,\dots,\alpha_N)=0$.  In particular, if $g=0$ and $n=3$, then $\br{\alpha_1,\alpha_2,\alpha_3}=0$ unless $D=0$, which occurs if and only if $\sum_{i=1}^3\deg_W\alpha_i=2\hat c$.
\end{ax}
\begin{ax}\label{ax:symmetry} Symmetry:
Let $\sigma\in S_N$. Then
\[
\br{\alpha_1,\dotsc,\alpha_N}_{g,n}=\br{\alpha_{\sigma(1)},\dotsc,\alpha_{\sigma(n)}}_{g,n}.
\]
\end{ax}
The next few axioms relate to the degrees of line bundles $\mathscr{L}_1,\dotsc,\mathscr{L}_N$ endowing an orbicurve $\mathscr C$ with $k$ marked points $p_1,\dotsc,p_k$ and endowed with a \emph{$W$-structure}.  This means that for each monomial $\prod_{j=1}^N z_j^{a_{ij}}$ of $W$, $\bigotimes_{j=1}^N \mathscr L_j^{\otimes{a_{ij}}}\cong \omega_\text{log}$.  Here, $\omega_\text{log}$ is the canonical bundle of $\mathscr C \backslash \{p_1,\dotsc,p_k\}$, and the identification of monomials in the $\mathscr L_j$ with $\omega_\text{log}$ arises naturally in the attempt to solve the Witten equation on the orbicurve $\mathscr C$.  The details may be found in \cite{FJR1} and provide geometric background to the present construction.

Consider the class $\Lambda_{g,k}^W(\alpha_1,\alpha_2,\dots,\alpha_k)$, with $\alpha_j\in(\mathscr{H}_{h_j})^G$ for each $j\in\{1,\dots,N\}$. For each variable $x_j$, then $\lb_j:=\deg\mathscr|\mathscr L_j|$ is given by
\begin{equation}\label{eq:Definition of line bundle degrees}
\lb_j=q_j(2g-2+k)-\sum_{i=1}^k\Theta_j^{h_i}.
\end{equation}
($|\mathscr L_j|$ denotes the pushforward of a bundle on the orbicurve $\mathscr C$ to the underlying coarse curve).
\begin{ax}\label{ax:integer degrees}
Integer degrees: If $\lb_j\notin \Z$ for some $j\in\set{1,\dots,N}$, then $\Lambda_{g,k}^W(\alpha_1,\alpha_2,\dots,\alpha_k)=~0$.
\end{ax}
\begin{remark}
  This axiom has the following important consequence, which follows immediately from examining Equation \eqref{eq:Definition of line bundle degrees}.
\end{remark}
\begin{corollary}\label{cor:multiplication lands in single sector}
Suppose $\Lambda_{g,k}(\alpha_1,\dotsc,\alpha_{k-1},\alpha_k)\neq 0$, with $\alpha_i\in\FJRW_{h_i}$.  Then
\[\Lambda_{g,k}(\alpha_1,\dotsc,\alpha_{k-1},\tilde \alpha_k) = 0 \text{ for any } \tilde \alpha_k\notin\FJRW_{h_k}.\]
\end{corollary}
\begin{proof}
Since
$\Lambda_{g,k}(\alpha_1,\dotsc,\alpha_{k-1},\alpha_k)\neq 0$, we know that for all $j$
\[\lb_j=q_j(2g-2+k)-\sum_{i=1}^k\Theta_j^{h_i}\in \Z.\]
Suppose $\alpha_k\in\FJRW_{\tilde h_k}$, where $\tilde h_k = (\tilde h_k h_k^{-1})h_k$. In order to have
\[\tilde \lb_j=q_j(2g-2+k)-\sum_{i=1}^{k-1}\Theta_j^{h_i} - \Theta_j^{\tilde h_k}\in \Z,\]
we need $\Theta_j^{\tilde h_k h_k^{-1}}\in \Z$.

Now, by Axiom \ref{ax:integer degrees}, $\Lambda_{g,k}(\alpha_1,\dotsc,\alpha_{k-1},\tilde \alpha_k) = 0$ unless this holds for all $j$, which is equivalent to $\tilde h_k = h_k$.
\end{proof}
\begin{ax}\label{ax:concavity}
Concavity: If $\lb_{j}<0$ for all $j\in\set{1,2,3}$, then $\br{\alpha_1,\alpha_2,\alpha_3}=1$.
\end{ax}
The next axiom is related to the Witten map.  When  $H^0(\bigoplus_{\substack{j=1\\}}^N \mathscr L_j)$ and $H^1(\bigoplus_{\substack{j=1\\}}^N \mathscr L_j)$ have the same rank, the Witten map is given by:
\begin{align*}
\mathcal W: H^0(\bigoplus_{\substack{j=1\\}}^N \mathscr L_j)\rightarrow H^1(\bigoplus_{\substack{j=1\\}}^N \mathscr L_j)\\
\mathcal W=\left(\overline{\frac{{\partial W}}{{\partial x_1}}}, \overline{\frac{{\partial W}}{{\partial x_2}}}, \dots, \overline{\frac{{\partial W}}{{\partial x_N}}}\right).
\end{align*}
Put $h^i_j=\text{rank}\, H^i(\mathscr L_j)$.

The fact that the Witten map is well-defined is a consequence of the geometric conditions on the $\mathscr{L}_j$ considered in \cite{FJR1}.  For further details, we refer readers to the original paper.

If $\Lambda_{g,n}^W(\alpha_1,\dots,\alpha_N)$ is a class of codimension zero, we obtain a complex number by integrating over $\overline{\mathscr{M}}_{g,n}$.  Abusing notation, we will refer to the class $\Lambda_{g,n}^W(\alpha_1,\dots,\alpha_N)$ and its integral over $\overline{\mathscr{M}}_{g,n}$ interchangeably.
\begin{ax}\label{ax:index-zero}
Index-Zero: Consider the class $\Lambda_{g,n}^W(\alpha_1,\alpha_2,\dots,\alpha_N)$, with $\alpha_i\in \h{\gamma_i}{G}$. If $\Fix \gamma_i=\set 0$ for each $i\in\set{1,2,\dots,n}$ and $\Lambda_{g,n}(\alpha_1,\alpha_2,\dots,\alpha_N)$ is of codimension
\[
\sum_{j=1}^N (h^0_j - h^1_j) = 0,
\]
then $\Lambda_{g,n}^W(\alpha_1,\alpha_2,\dots,\alpha_N)$ is equal to the degree of the Witten map.
\end{ax}
\begin{ax}\label{ax:composition}
Composition: If the four-point class, $\Lambda_{g,n}^W(\alpha_1,\alpha_2,\alpha_3,\alpha_4)$ is of codimension zero, then the correlator $\br{\alpha_1,\alpha_2,\alpha_3,\alpha_4}$ decomposes in terms of three-point correlators in the following way:
\[
\br{\alpha_1,\alpha_2,\alpha_3,\alpha_4}=\sum_{\beta,\delta}\br{\alpha_1, \alpha_2,\beta}\eta^{\beta,\delta}\br{\delta,\alpha_3,\alpha_4}.\]
\end{ax}
Note that $\Fix J=\set 0$ so $\mathscr{H}_{J}\cong\C$ and $\deg \FJRW_J=0$. The identity element in the FJRW-ring is an element of $\mathscr{H}_{J}$, and we denote this element by $\1$.
\begin{ax}\label{ax:pairing}
Pairing: For $\alpha_1,\alpha_2\in \h WG$,  $\br{\alpha_1,\alpha_2,\mathbf{1}}=\eta(\alpha_1,\alpha_2)$, where $\eta$ is the pairing in $\mathscr{H}_{W,G}$.
\end{ax}
\begin{ax} \label{ax:sums} Sums of singularities: If $W_1\in\C[x_1,\dots,x_r]$ and $W_2\in\C[y_1,\dots,y_s]$ are two non-degenerate, quasi-homogeneous polynomials with maximal symmetry groups $G_1$ and $G_2$, then the maximal symmetry group of $W=W_1+W_2$ is $G=G_1\times G_2$, and there is an isomorphism of Frobenius algebras
\[
\h{W}{G}\cong \h{W_1}{G_{W_1}}\otimes \h{W_2}{G_{W_2}}
\]
\end{ax}
\begin{remark}\label{rem:sums}
We note an important consequence of Axiom \ref{ax:sums}.  Under the same hypotheses as in the statement of the axiom, we have a Frobenius Algebra isomorphism
\[\Q_{W}\cong \Q_{W_1}\otimes \Q_{W_2},\]
and similarly
\[\Q_{W^T}\cong \Q_{W_1^T}\otimes \Q_{W_2^T}.\]
Consequently, in order to prove the Mirror Symmetry Conjecture for $W=W_1+W_2$ a sum of decoupled polynomials (with maximal $A$-model orbifold group, dual to the trivial $B$-model orbifold group), it suffices to prove it for $W_1$ and $W_2$ individually.
\end{remark}
\begin{ax} Deformation Invariance: \label{ax:deformation invariance}
$\Lambda_{g,n}^W(\alpha_1,\alpha_2,\dots,\alpha_N)$ is independent of the representative of $W$.
\end{ax}
\subsection{Orbifold $B$-model}\label{sec:B-model}

Let $W\in\C[y_1,\dotsc,y_N]$ be a non-degenerate quasi-homogeneous polynomial where $y_i$ has weight $q_i\in \mathbb{Q}$.

We will take $W$ to be an invertible potential, so $W = \sum W_j$ where each $W_j\in\C[y^{(j)}_1,\dotsc,\,y^{(j)}_{n_j}]$ is of loop, chain, or Fermat type.

Let $G\subset (\C^*)^N$ be a group of diagonal symmetries of $W$.

For $g\in G$, $\Fix(g)=\C^{N_g}$ where $N_g = \dim \Fix(g)$.  Put $\Q_g:=\Q_{W|_{\Fix g}}\omega_{\Fix g}$, where as before the presence of the volume form $\omega_{\Fix g}$ encodes a determinant twist of the natural $G$-action of $\Q_{W|_{\Fix g}}$.

\begin{definition}
  The unprojected state space of the Landau--Ginzburg orbifold $B$-model of $W/G$ is defined to be
  \[\Q =\bigoplus_{g\in G}\Q_g.\]
\end{definition}

This defines $\Q$ as a $G$-graded $\C$-vector space.  $\Q$  also possesses a $\mathbb{Q}$ bi-grading, which we discuss in the next section.  We will show that the multiplication defined in this section respects the bi-grading.

We endow $\Q$ with a non-degenerate pairing $\langle\,,\,\rangle$ by taking the sum of the pairings $\Q_g\tensor\Q_{g^{-1}}\to \C$, which are induced by the residue pairing under the identification $\Q_g\cong \Q_{g^{-1}}$.

We aim to endow $\Q$ with an algebra structure which preserves both the $G$-grading and the $\mathbb{Q}$ bi-grading.  We observe that for $g \in G$, we have a ring homomorphism $\Q_e\to \Q_g$ given by setting variables not fixed by $g$ equal to zero.  This induces on $\Q_g$ the structure of a cyclic $\Q_e$ module, with $1\in \Q_g$ as the generator of the $g$-graded summand.

So to define an algebra structure on $\Q$, it suffices to define a compatible multiplication
\[1_g \star 1_h = \gamma_{g,h} 1_{gh}.\]
Since $1_e$ will be the identity for the multiplication, we require
\begin{equation}
1_e\star 1_g = 1_g \phantom{X}\text{ so }\phantom{X}\gamma_{e,g} = 1_g = \gamma_{g,e}.
\label{eq:B model identity axiom}
\end{equation}
For the multiplication to be associative, we must have
\begin{equation}
(1_g\star 1_h)\star 1_k = 1_g\star (1_h\star 1_k)\phantom{X}\text{ so }\phantom{X}
\gamma_{g,h}\gamma_{gh,k} = \gamma_{g,hk}\gamma_{h,k}.
\label{eq:B model associativity axiom}
\end{equation}
We propose the following definition of $\gamma$ and check that it satisfies \eqref{eq:B model identity axiom} and \eqref{eq:B model associativity axiom}.
\begin{definition}\label{def:B-model multiplication}
For $g\in G$, let $I_g = \{i\,:\, g_i = 1\}$.  Define $\gamma$ through the equation
\begin{equation}\label{eq:B-model explicit cocycle definition}
\gamma_{g,h} \frac{\hess W|_{\Fix(g)\cap \Fix{h}}}{\dim\Fix(g)\cap\Fix(h)} =
\begin{cases}
\frac{\hess W|_{\Fix{gh}}}{\dim\Fix (gh)} &  \text{ if } I_g\cup I_h\cup I_{gh} = \{1,\cdots, n\}\\
0 &\text{else}.
\end{cases}
\end{equation}
\end{definition}
\begin{remark}
By definition, $\gamma_{g,h}$ has non-zero pairing with the determinant of the hessian of $W$ on the common fixed locus of $g$ and $h$, provided each variable is fixed by at least one of $g$, $h$ and $gh$.  The factor of $\dim \Fix(g)\cap \Fix(h)$ in the denominator ensures that Condition \eqref{eq:B model identity axiom} is satisfied.
\end{remark}
\begin{proposition}
  The above multiplication $\star$ is associative.
\end{proposition}
\begin{proof}
This definition obviously satisfies \eqref{eq:B model identity axiom}, since $1\in\Q_{g}$ pairs to unity with $\hess W|_{\Fix(g)}$.

It remains to check the associativity \eqref{eq:B model associativity axiom} of the putative cocycle $\gamma$.

We see here the benefit of restricting our attention to invertible potentials (sums of loops, chains, and Fermat types).

We first check associativity of multiplication when $W$ is of one of these atomic types.  The key point here is that if $k\in(\C^*)^N$ is a symmetry of $W$ fixing $y_1$, then $k$ acts trivially on all of $\C^N$.  So $1_g\star 1_h = \gamma_{g,h}1_{gh}$ can be non-zero only if one of $g$, $h$, or $gh$ is the identity.

If $g=\id$, $h=\id$, or $k=\id$ then associativity is obvious.

Suppose $g\neq \id$, $h\neq \id$ and $k\neq\id$.   We show that both sides of \eqref{eq:B model associativity axiom} vanish.  Consider the left hand side. If $gh\neq \id$ then by the above remark, $1_g\star 1_h=0$.  If $gh = \id$, the left hand side is $\gamma_{g,g^{-1}}1_k$. Now,  $\gamma_{g,g^{-1}}$ pairs with $\hess W|_{\Fix(g)}$, so depends on the variables not fixed by $g$ (in particular $y_1$).  Since $k\neq\id$, $y_1$ is not fixed by $k$, and $\gamma_{g,g^{-1}}1_k=0\in\Q_k$.  A similar argument applies to the right hand side.

Thus we have an associative multiplication on $\Q$ for $W$ a loop, chain, or Fermat potential.  In fact, we have shown furthermore that a triple-product vanishes unless one of the factors is in the identity sector, and the other two factors are in sectors corresponding to mutually inverse group elements.

This multiplication (Definition \ref{def:B-model multiplication}) extends to any invertible potential, as the product may be decomposed into contributions from each atomic summand, and associativity on the summands implies associativity for the whole invertible potential.
\end{proof}
In the next section, we show that the multiplication on the unprojected state space descends to a multiplication on invariants, without making any assumptions about the potential being of atomic type.
\subsubsection{Projecting to invariants}
Now we turn our attention to the $G$-invariants in $\Q$ for the determinant-twisted $G$ action.  We make the important restriction that $G\subseteq SL_N\C$, so that the $G$-invariants in $\Q_e$ are the same whether or not we twist by the determinant on $\Fix\id = \C^N$.  This means that the $\Q_e$-module structure on $\Q=\bigoplus_{g\in G} \Q_g$ descends to a $(\Q_e)^G$-module structure on the determinant-twisted $G$ invariants $\left(\bigoplus_{g\in G}\Q_g\right)^G$.  This hypothesis will be justified later when we see that admissible $A$-model orbifold groups correspond to subgroups of $SL_N\C$ on the $B$-side.

To see that the product descends to invariants, we prove the following lemma.
\begin{lemma}\label{lem:B-model projecting to invariants}
Suppose $H, K\in Q_e$ are monomials such that $H 1_h\in\Q_h$ and $K1_k\in\Q_k$ are (determinant-twisted) $G$-invariants.  Then $HK 1_h\star 1_k$ is a (determinant-twisted) $G-invariant$.
\end{lemma}
\begin{proof}
The lemma is trivially true if $HK 1_h\star 1_k = 0$.  We may therefore suppose that for each $i\in\{1,\dotsc, n\}$, at least one of $h_i$, $k_i$ or $h_i k_i$ equals $1$.

$G$-invariance of the $H1_h$ and $K1_k$ yields
\begin{gather}
g(H) \prod_{i\in I_h} g_i = 1 \label{eq:H1_h G-invariant}\\
g(K) \prod_{i\in I_k} g_i = 1, \label{eq:K1_k G-invariant}
\end{gather}
where $g(H)$ denotes the phase of the action of $g$ on the monomial $H$, and similarly for $g(K)$.
We need to compute the action of $g$ on $HK 1_h\star 1_k$.

Since we assume $1_h\star 1_k \neq 0$, Equation \eqref{eq:B-model explicit cocycle definition} applies.  The phase of $g$ on either side of this relation must coincide, so
\[g(\gamma_{h,k}) = \frac{\prod_{i\in I_{hk}}g_i^{-2}}{\prod_{i\in I_h\cap I_k}g_i^{-2}}.\]

Then, using \eqref{eq:H1_h G-invariant} and \eqref{eq:K1_k G-invariant}, the phase of $g$ on
\[(H1_h)\star(K1_k) = HK\gamma_{h,k}1_{hk}\]
is
\begin{multline*}g(H)g(K)g(\gamma_{h,k})\prod_{i\in I_{hk}} g_i
= \prod_{i\in I_h} g_i^{-1}\prod_{i\in I_k} g_i^{-1}\frac{\prod_{i\in I_h\cap I_k}g_i^{2}}{\prod_{i\in I_{hk}}g_i^{2}}\prod_{i\in I_{hk}} g_i
= \prod_{i\in I_h\cup I_k} g_i^{-1}\prod_{i\in I_h\cap I_k}g_i\prod_{i\in I_{hk}}g_i^{-1}\\
= \prod_{i\in\{1,\dotsc,\,n\}} g_i^{-1}
=1
\end{multline*}
by the assumption $G\subseteq SL_N\C$.

So the $\star$-product of $G$-invariants is again $G$-invariant.
\end{proof}

\subsubsection{Pairing and Frobenius Algebra}
The pairing $\langle\,,\,\rangle$ on $\Q_{W,G}$ is the sum of the pairings $\Q_g\tensor\Q_{g^{-1}}\to \C$, which are induced by the residue pairing under the identification $\Q_g\cong \Q_{g^{-1}}$.

The orbifold Milnor ring (after projecting to $G$ invariants) is a Frobenius Algebra.  This follows from the definition of the pairing and the associativity of multiplication.

By construction, the above multiplication preserves the $G$-grading, and we will show in the next section that it preserves the $\mathbb{Q}$ bi-grading also.

\subsection{Bi-grading}\label{sec:Bi-grading}

Recall the Intriligator--Vafa grading \cite{IV}:
\[J_{\pm} = \pm \sum_{\Theta_i^h\notin\Z} (\Theta_i^h -\tfrac{1}{2}) + \sum_{\Theta_i^h\in\Z} (q_i-\tfrac{1}{2}),\]
and \[\hat{c} = \sum(1 - 2q_i)\]
We introduce the following bi-gradings for Landau--Ginzburg Theories, for a sector corresponding to a symmetry $h=(e^{2\pi i\Theta_1^h},\dotsc,e^{2\pi i\Theta_N^h})$ of the potential $W\in\C[X_1,\dotsc,X_N]$ having charges $q_1,\dotsc,q_N$:

\begin{equation}\label{eq:bigrading}
\begin{array}{|c||r@{\hspace{10pt} = \hspace{10pt}}l|}\hline
Q^A_+&J_+ + \tfrac{\hat{c}}{2}& \sum_{\Theta^h_i\notin \Z}(\Theta_i^h - q_i)\\\hline
Q^A_-&-J_- + \tfrac{\hat{c}}{2}&\sum_{\Theta^h_i\notin \Z}(\Theta_i^h - q_i) + \sum_{\Theta_i^h\in\Z}(1-2q_i)\\\hline
\multicolumn{3}{c}{}\\\hline
Q^B_+&J_+ + \tfrac{\hat{c}}{2}&\sum_{\Theta^h_i\notin \Z}(\Theta_i^h - {q}_i)\\\hline
Q^B_-&J_- + \tfrac{\hat{c}}{2}&\sum_{\Theta_i^h\notin\Z} (1 - \Theta_i^h - {q}_i)\\\hline
\end{array}
\end{equation}
\begin{remark}
The grading above is an `external' grading.  The $A$ and $B$ models have `internal' gradings coming from the weighted degree of monomials in the Milnor rings which are summands in the state space.  A monomial of weighted degree $p$ has an $A$-model bi-grading of $(p,-p)$ and a $B$-model bi-grading of $(p,p)$.

Note
\begin{equation}\label{eq:bigrading+-}
\begin{array}{|c||r@{\hspace{10pt} = \hspace{10pt}}l|}\hline

Q^A_+ + Q^A_-&J_+ - J_- +\hat{c}& \dim\Fix h + 2\sum(\Theta_i^h-q_i)\\\
Q^A_+ - Q^A_-&J_+ + J_- & \sum_{\Theta_i^h\in\Z}(2q_i - 1)
\\\hline

\multicolumn{3}{c}{}\\\hline
Q^B_+ + Q^B_-&J_+ + J_- +\hat{c}& \sum_{\Theta_i^h\notin\Z}(1-2{q}_i)\\
Q^B_+ - Q^B_-&J_+ - J_- &2\sum_{\Theta_i^h\notin\Z}(\Theta_i^h-{q}_i) + \sum_{\Theta_i^h\notin\Z}(2{q}_i - 1)\\\hline
\end{array}
\end{equation}

So this grading recovers the $A$-model grading of \cite{FJR1} as the sum of the $A$-model bi-gradings.  The internal $A$-model bi-grading has
following interpretation. $H^{N}(\C^N, W^{\infty}, {\C})$ has a mixed Hodge structure which defines Hodge grading and Hodge decompostion
$$H^{N}(\C^N, W^{\infty}, {\C})=\bigoplus_p H^{p, N-p}.$$
Under the isomorphism to the  Milnor ring, $H^{p, N-p}$ corresponds to the degree $p$-component of Milnor ring. We have absorbed $N$ into
external grading, so the internal bi-grading is $(p, -p)$.

It is desirable to show that the \emph{difference} of the $A$-model bi-gradings (corrected by twice the internal grading) is preserved under $A$-model multiplication.  In full generality, this has been intractable because it would demand a more precise understanding of the Ramond sector contribution to $A$-model multiplication than is currently available.  However, in the case of maximal $A$-model symmetry group, the desired fact follows from Theorem \ref{thm:Mirror Theorem for G_max} and the fact that the mirror map of Section \ref{sec:Mirror Map} preserves the bi-grading.
\end{remark}
\begin{lemma}
  The $B$-model multiplication (Def. \ref{def:B-model multiplication}) preserves the bigrading $(Q^B_+, Q^B_-)$ of Eq. \eqref{eq:bigrading}.
\end{lemma}
\begin{proof}
To show that the multiplication on the orbifold $B$-model respects the bi-grading \eqref{eq:bigrading}, it suffices to show it respects the bi-grading \eqref{eq:bigrading+-}, accounting of course for the internal $(p,p)$ grading  on the $B$-model.

Because the $B$-model is a module over $\Q_e$, the contribution of the internal grading is obviously additive under multiplication.  The contribution of
\begin{equation}\label{eq:degB+}
\deg_+^B:=\sum_{\Theta_i^h\notin\Z}(1-2q_i)
\end{equation} is additive under multiplication because if $(1_g\star 1_h)\neq 0$,
\begin{align*}
\deg_{+}^B(1_g\star 1_h)&= Q^B_++Q^B_- + 2\deg(\hess W|_{\Fix{gh}})-2\deg(\hess W|_{\Fix(g)\cap\Fix(h)})\\
&=  \sum_{\Theta_i^{gh}\notin\Z}(1-2q_i) +2 \sum_{\Theta_i^{gh}\in\Z}(1-2q_i) -2\sum_{\Theta_i^{g},\Theta_i^h\in\Z}(1-2q_i),
\end{align*}
while
\begin{align*}
\deg_{+}^B(1_g) + \deg_{+}^B(1_h) &=  \sum_{\Theta_i^{g}\notin\Z}(1-2q_i)+\sum_{\Theta_i^{h}\notin\Z}(1-2q_i).
\end{align*}
We check by cases that these are equal (in fact in the two sums, the summand corresponding to the $i^\text{th}$ variable is the same).
\begin{itemize}
\item If $\Theta_i^g\notin\Z $ and $\Theta_i^h\notin \Z$, then $\Theta_i^{gh} \in\Z$ since every variable is fixed by $g$, $h$, or $gh$ for non-zero $B$-model product.  So the contribution to the right of each equation is $2(1-2q_i)$.
\item If $\Theta_i^{g} \in\Z$ and $\Theta_i^{h} \notin\Z$ (or vice versa), then $\Theta_i^{gh} \notin\Z$, and the contribution to the right of each equation is $1-2q_i$.
\item If $\Theta_i^{g} \in\Z$ and $\Theta_i^{h} \in\Z$, then the contribution to the right of each equation is $0$.
\end{itemize}

We now check that the $B$-model multiplication respects
\begin{equation}\label{eq:degB-}
deg_{-}^B :=2\sum_{\Theta_i^h\notin\Z}(\Theta_i^h-{q}_i) + \sum_{\Theta_i^h\notin\Z}(2{q}_i - 1)=\sum_{\Theta_i^h\notin\Z}(2\Theta_i^h-1) .
\end{equation}
\begin{remark}
The internal grading doesn't contribute to $\deg_{-}^B$, as the difference of the $(p,p)$ bi-grading is zero).
\end{remark}

Now
\[
\deg_{-}^B(1_g\star 1_h)= \sum_{\Theta_i^{gh}\notin\Z}(2\Theta_i^{gh}-1),
\]
and
\[
\deg_{-}^B(1_g) + \deg_{-}^B(1_h) =\sum_{\Theta_i^{g}\notin\Z}(2\Theta_{i}^g-1)+\sum_{\Theta_i^{h}\notin\Z}(2\Theta_i^h-1).
\]
Again the two sums match up term by term:
\begin{itemize}
\item If $\Theta_i^g\notin\Z $ and $\Theta_i^h\notin \Z$, then $\Theta_i^{gh} \in\Z$ since every variable is fixed by $g$, $h$, or $gh$ for non-zero $B$-model product.  Therefore $\Theta_i^g+\Theta_i^h=1$, and the contribution to the right of each equation is 0.
\item If $\Theta_i^{g} \in\Z$ and $\Theta_i^{h} \notin\Z$, then $\Theta_i^{gh} = \Theta_i^h\notin\Z$, and the contribution to the right of each equation is $2\Theta_i^h-1$.  Similarly if
$\Theta_i^{h} \in\Z$ and $\Theta_i^{g} \notin\Z$.\item If $\Theta_i^{g} \in\Z$ and $\Theta_i^{h} \in\Z$, then the contribution to the right of each equation is $0$.
\end{itemize}
\end{proof}
\subsection{Relation between $A$ and $B$ model for a fixed singularity}
Note that the state spaces of the $A$ and $B$ models for a fixed singularity are isomorphic as vector spaces.  For its bi-grading,
$$deg^A_+=deg^B_+,\phantom{XX} deg^A_-=-deg^B_-+\hat{c}.$$
This simple relation is  particularly relevent in the Calabi-Yau case ($\sum q_i = 1$) where the same relation holds for the Calabi-Yau
hypersurface defined by $W=0$, giving further evidence of Landau--Ginzburg mirror symmetry.
\section{Mirror Symmetry for State Spaces}\label{sec:Mirror Symmetry}
\subsection{Duality of Groups}\label{sec:Duality of Groups}
Following Berglund--H\"ubsch \cite{BH}, we consider the transposed singularity
\[W^T= \sum_{i=1}^N \prod_{j=1}^N Y_j^{a_{ji}},\]
which has exponent matrix $A^T$.

This suggests writing \begin{equation}\label{eq:A inverse rows = rho bar}
A^{-1} =
\left(\begin{array}{c}
\overline\rho_1 \\\hline
\overline\rho_2 \\\hline
\vdots \\\hline
\overline\rho_N
\end{array}\right),
\phantom{XX}
\text{ with row vectors }
\phantom{XX}
\overline\rho_i=
\begin{bmatrix}
\overline\varphi^{(i)}_1 & \cdots & \overline\varphi^{(i)}_N
\end{bmatrix}.
\end{equation}
Comparing to Equation \eqref{eq:A inverse columns = rho}, we see $\overline\varphi^{(i)}_j = \varphi^{(j)}_i$.

As above, each $\overline\rho_k$ is a symmetry of $W^T$ and generate $G^\text{max}_{W^T}$, where
\[\overline\rho_kY_j = \exp(2\pi i\overline\varphi^{(k)}_j)Y_j,\]
and the exponential grading operator is
$\overline J = \prod_{i=1}^N \overline \rho_i$.

The following lemma is straightforward, but essential to what follows.
\begin{lemma}\label{lem:dual group corresponds to invariant monomials}
\[\prod_{i=1}^N \rho_i^{\alpha_i} \text{ preserves the monomial } \prod_{j=1}^N X_j^{r_j}\]
 if and only if
\[\prod_{j=1}^N \overline \rho_j^{r_j} \text{ preserves the monomial } \prod_{i=1}^N Y_i^{\alpha_i}.\]
\end{lemma}
\begin{proof}
Both statements are equivalent to
\[\begin{bmatrix}
r_1, & \dotsc &, r_N
\end{bmatrix}
A_W^{-1}
\begin{bmatrix}
\alpha_1\\
\vdots\\
\alpha_N
\end{bmatrix}
\in \Z\]
\end{proof}
In particular, since each $\overline\rho_k$ preserves every monomial $\prod_{i}Y_i^{a_{ij}}$ appearing in $W^T$, we have that
$\prod_{i} \rho_i^{a_{ij}}$ preserves every $X_k$.  That is:
\begin{corollary}\label{lem:Skarke Observation}
Let $W = \sum_{i=1}^N\prod_{j=1}^N X_j^{a_{ij}}$ be a non-degenerate, invertible singularity with exponent matrix $A=(a_{ij})$.  Let the symmetry $\rho_k$ be given by the $k$th column of $A^{-1}$ as above.

Then \[\prod_{i=1}^N \rho_i^{a_{ij}} = 1\]
for every $j\in\{1,\dotsc, n\}$.
\end{corollary}

\begin{remark}
In \cite{Kreuzer}, this observation is attributed to Skarke in the special case of `Loop potentials'.
\end{remark}
\begin{definition}\label{def:Dual group definition}
We define dual group $G^T$ as
\begin{equation}
G^T := \left\{\prod_{i=1}^N \overline \rho_i^{r_i} \left|  \, [r_1, \cdots, r_N] A^{-1}_W \begin{bmatrix}
a_1\\
\vdots\\
a_N
\end{bmatrix}\in {\Z}
\text{ for all } \prod_{i=1}^N\rho_i^{a_i}\in G
\right.\right\}.
\end{equation}
\end{definition}
If $g=\prod_{i=1}^N \rho^{a'_i}_i$ is a different presentation, $\prod_{i=1}^N \rho^{a_i-a'_i}_i=1$. Hence,
\[A^{-1}_W [a_1-a'_1, \cdots, a_N-a'_N]^T\in \Z^{N}.\]
Therefore, the above definition is independent of presentation of elements of $G$.
\begin{lemma}\label{lem:(G^T)^T = G}
Let $G$ be a group of diagonal symmetries of the non-degenerate invertible potential $W$, and $G^T$ the dual group of symmetries of $W^T$.  Then
	\[(G^T)^T = G.\]
\end{lemma}
\begin{proof}
  It is clear from the definition that $G\subseteq (G^T)^T$ and $\C[X_1,\dotsc,X_N]^G \subseteq \C[X_1,\dotsc,X_N]^{(G^T)^T}$.  This implies that $G$ and $(G^T)^T$ have equal invariant rings, and since the actions on $\C[X_1,\dotsc, X_N]$ extend to actions on the fraction field with the same fixed field, it follows from field theory that $G=(G^T)^T$.
\end{proof}
It is also obvious ${1}^T=G^\text{max}$. Now we compute $\br{J}^T$.
Since $J=\prod_{i=1}^N\rho_i$,  $h=\prod_{i=1}^N \bar{\rho}^{r_i}_i \in \br{J}^T$ if and only if $\sum_i r_i q_i\in \Z$.  Since $\sum_i r_i q_i$ is
precisely the phase of $\det(h)$,  we have $\br{J}^T=SL_N\C\cap G^\text{max}_{W^T}$.

This explains the $SL$ restriction made in the proof of Lemma \ref{lem:B-model projecting to invariants} that the orbifold $B$-model multiplication descends to the invariants under the action of the orbifold group.

We can use the argument from the proof of \ref{lem:(G^T)^T = G} to settle a question suggested in \cite{FJR1}, namely whether any diagonal symmetry group containing $J$ satisfies the following definition of admissible groups.
\begin{definition}[\cite{FJR1} Defn 2.3.2]
We say that a subgroup $G\leq G_W^\text{max}$ is \emph{admissible} or is an \emph{admissible group of Abelian symmetries of $W$} if there exists a Laurent polynomial $Z$, quasi-homogeneous with the same weights $q_i$ as $W$, but with no monomials in common with $W$, and such that $G= G_{W+Z}$.
\end{definition}
\begin{proposition}
For $W\in\C[X_1,\dotsc,X_N]$ a non-degenerate (not-necessarily invertible) potential, any group of diagonal symmetries of $W$ containing $J$ is admissible.
\end{proposition}
\begin{proof}
For a group $G$ of diagonal symmetries of $W$ containing $J$ to be admissible, we require the existence of a Laurent polynomial $Z$ in $X_1,\dotsc,X_N$, quasi-homogeneous with the same weights as $W$,  such that $G$ is the \emph{maximal} diagonal symmetry group of $W+Z$.

Now, the ring of $G$-invariants is finitely generated by monomials.  If we let $Z$ be the sum of those $G$-invariant monomials not divisible by monomials in $W$, $G$ is the maximal diagonal symmetry group of $W+Z$.  (Otherwise there is a diagonal symmetry group $H$, with $G\subseteq H$ and $\C[X_1,\dotsc,X_N]^G\subseteq \C[X_1,\dotsc,X_N]^H$, implying $G=H$ as before).   Since $J$ preserves each of the constituent monomials of $Z$, each of these monomials has integral quasi-homogeneous degree.  We may correct each of these monomials by a (negative) power of any monomial in $W$ to ensure that each of the monomials has quasi-homogeneous degree equal to $1$, and since we are correcting by $G$-invariants not dividing the monomials of $Z$, we do not change the maximal symmetry group of $W+Z$.
\end{proof}
\subsection{Mirror Map}\label{sec:Mirror Map}
We propose in this section a `Mirror map' $\Q_{W^T,G^T}\to\h WG$.

\begin{definition}[Mirror Map]\label{def:Invertible Potential Mirror Map}
Let $W$ be a non-degenerate, invertible potential and $G$ and admissible $A$-model diagonal symmetry group of $W$.  Define the linear map $\Q_{W^T,G^T}\to\h WG$ by the following map on generators
\begin{equation}
\begin{split}\label{eq:Invertible Potential Mirror Map}
\left(\bigoplus_{g\in {G^T}}\mathscr{Q}^{W^T}_g\right)^{G^T}&\longrightarrow\left(\bigoplus_{g \in G} \mathscr{H}^W_g\right)^G\\
\prod Y_j^{\alpha_j}dY_j \ket{\prod\overline \rho_j^{r_j+1}}&\longmapsto \prod X_j^{r_j}dX_j \ket{\prod \rho_j^{\alpha_j+1}},
\end{split}
\end{equation}
where it should be understood that the range of the product over $X$'s is the same as the range of the product over the $\overline\rho$'s, and similarly for the $Y$'s and the $\rho$'s.
\end{definition}
\begin{remark}
One may worry that this map is not well-defined, since the monomial $\prod X_j^{r_j}dX_j$ may not be uniquely determined by an element of $G^T$.  However, this monomial \emph{is} completely determined given both the group element \emph{and} the locus over which to extend the product, namely the fixed locus of the element of $G$ corresponding to $\prod Y_j^{\alpha_j}dY_j$.

We shall see how this correspondence arises in the proof of Theorem \ref{thm:Mirror Theorem for State Spaces}.
\end{remark}
\begin{example}
We present here the example of the two-variable loop potential $W=x^3y+xy^5$, orbifolded by $J=(e^{(2\pi i){2}/{7}},e^{(2\pi i){1}/{7}})$ on the $A$-side and by the dual group $J^T=\br{(-1,-1)}$ on the $B$-side.  The table below presents the vector space generators for the $A$-model $W/J$ and the $B$-model $W^T/J^T$, along with the bigrading.  We denote the standard volume form on $\Fix\rho_x^a\rho_y^b$ by $e_{\rho_x^a\rho_y^b}$. The $A$ and $B$ model invariants in each column correspond to each other under the Mirror Map (Equation \eqref{eq:Invertible Potential Mirror Map}), and evidently the bi-grading is preserved.
\begin{center}
\begin{tabular}{|c||c|c|c|c|c|c|c|c|c|}\hline
$W/J$	& $e_{\rho_x^1\rho_y^1}$	&$e_{\rho_x^2\rho_y^2}$	&$e_{\rho_x^0\rho_y^2}$	&$e_{\rho_x^1\rho_y^3}$	&$e_{\rho_x^2\rho_y^4}$	 &$e_{\rho_x^0\rho_y^4}$ & $x^2e_{\rho_x^0\rho_y^0}$ & $xy^2e_{\rho_x^0\rho_y^0}$ & $y^4e_{\rho_x^0\rho_y^0}$\\\hline\hline
$\deg_+^A$	&0 &$\tfrac{6}{7}$	&$\tfrac{12}{7}$ &$\tfrac{4}{7}$	&$\tfrac{10}{7}$	&$\tfrac{16}{7}$	&$\tfrac{8}{7}$ &$\tfrac{8}{7}$ &$\tfrac{8}{7}$\\\hline
$\deg_-^A$	&0	&0  &0  &0  &0	&0	&0	&0	&0\\\hline\hline
$W^T/SL$	& $e_{\overline\rho_x^0\overline\rho_y^0}$	&$xye_{\overline\rho_x^0\overline\rho_y^0}$	&$x^2y^2e_{\overline\rho_x^0\overline\rho_y^0}$	& $y^2e_{\overline\rho_x^0\overline\rho_y^0}$	&
$xy^3e_{\overline\rho_x^0\overline\rho_y^0}$ & $y^4e_{\overline\rho_x^0\overline\rho_y^0}$	&
$x^2e_{\overline\rho_x^3\overline\rho_y^1}$	&$e_{\overline\rho_x^2\overline\rho_y^3}$	&$y^4e_{\overline\rho_x^1\overline\rho_y^5}$\\\hline	
$\deg_+^B$	&0 &$\tfrac{6}{7}$	&$\tfrac{12}{7}$ &$\tfrac{4}{7}$	&$\tfrac{10}{7}$	&$\tfrac{16}{7}$	&$\tfrac{8}{7}$ &$\tfrac{8}{7}$ &$\tfrac{8}{7}$\\\hline
$\deg_-^B$	&0	&0  &0  &0  &0	&0	&0	&0	&0\\\hline
\end{tabular}
\end{center}
\end{example}
\begin{example}
Now we present the example of the two-variable chain potential $W=x^3y+y^4$, orbifolded by $J=(e^{2\pi i/4},e^{2\pi i/4})$ on the $A$-side and by the dual group $J^T=\br{(e^{2\pi i /3},e^{2\pi i/6})}$ on the $B$-side.  The table below presents the vector space generators for the $A$-model $W/J$ and the $B$-model $W^T/J^T$, along with the bigrading.  The $A$ and $B$ model invariants in each column correspond to each other under the Mirror Map (Equation \eqref{eq:Invertible Potential Mirror Map}), and we see again that the bi-grading is preserved.
\begin{center}
\begin{tabular}{|c||c|c|c|c|c|c|}\hline
$W/J$	& $e_{\rho_x^1\rho_y^1}$	&$e_{\rho_x^2\rho_y^2}$	&$e_{\rho_x^0\rho_y^3}$	&$x^2e_{\rho_x^0\rho_y^0}$	&$xye_{\rho_x^0\rho_y^0}$	 &$y^2e_{\rho_x^0\rho_y^0}$\\\hline\hline
$\deg_+^A$	&0	&1 &2	&1	&1	&1\\\hline
$\deg_-^A$	&0	&0	&0	&0	&0	&0\\\hline\hline
$W^T/SL$	& $e_{\overline\rho_x^0\overline\rho_y^0}$	&$xye_{\overline\rho_x^0\overline\rho_y^0}$	&$x^2y^2e_{\overline\rho_x^0\overline\rho_y^0}$	&$
e_{\overline\rho_x^3\overline\rho_y^1}$	&$e_{\overline\rho_x^2\overline\rho_y^2}$	&$e_{\overline\rho_x^1\overline\rho_y^3}$\\\hline	
$\deg_+^B$	&0 &1	&2	&0	&1	&1\\\hline
$\deg_-^B$	&0	&0	&0	&0	&0	&0\\\hline
\end{tabular}
\end{center}
\end{example}
\subsection{Mirror Symmetry for State Spaces}\label{sec:Mirror Symmetry for State Spaces}
In this section, we prove that the Mirror Map (Equation \eqref{eq:Invertible Potential Mirror Map}) is a bi-degree preserving vector space isomorphism.
\begin{theorem}\label{thm:Mirror Theorem for State Spaces}
Let $W$ be a non-degenerate invertible potential and $G$ an admissible group of diagonal symmetries of $W$.   The Mirror Map defined on generators by
Equation \eqref{eq:Invertible Potential Mirror Map} is a bi-degree preserving isomorphism of vector spaces.
\end{theorem}
\begin{remark}
For $G=G_W^\text{max}$ and ignoring the bi-grading, this recovers the main result of \cite{Kreuzer}.
\end{remark}
The proof proceeds in two steps. First, we
consider the \emph{total unprojected mirror map} (Definition \ref{def:Invertible Potential Total Mirror Map}), which we prove to be a bi-degree preserving vector space isomorphism.

Second, we note that by definition of the dual group, this restricts to an isomorphism on invariants (under the $G$-action for the $A$ model and the $G^T$ action for the $B$-model).

\begin{definition}[Total Unprojected Mirror Map]\label{def:Invertible Potential Total Mirror Map}
Let $W$ be a non-degenerate, invertible potential and $G$ and admissible $A$-model diagonal symmetry group of $W$.  Define the linear map $\Q_{W^T,G^T}\to\h WG$ by the following map on generators
\begin{equation}
\begin{split}\label{eq:Invertible Potential Total Mirror Map}
\bigoplus_{g\in {G_{W^T}^\text{max}}}\mathscr{Q}^{W^T}_g&\longrightarrow \bigoplus_{g \in G_W^\text{max}} \mathscr{H}^W_g\\
\prod Y_j^{\alpha_j}dY_j \ket{\prod \overline\rho_j^{r_{j}+1}}&\longmapsto \prod X_j^{r_j}dX_j \ket{\prod \rho_j^{\alpha_j+1}},
\end{split}
\end{equation}
where the ambiguity in the range of the products is resolved as in Definition \ref{def:Invertible Potential Mirror Map}.
\end{definition}
\begin{theorem}\label{thm:Mirror Theorem for Total State Spaces}
Let $W$ be a non-degenerate invertible potential.  The Total Unprojected Mirror Map defined on generators by
Equation \eqref{eq:Invertible Potential Total Mirror Map} is a bi-degree preserving isomorphism of vector spaces.
\end{theorem}
\begin{proof}
Note that to show the total unprojected mirror map (Equation \eqref{eq:Invertible Potential Total Mirror Map}) is a degree-preserving isomorphism of vector spaces, it suffices to do so for all invertible potentials.  To see this, suppose $W=\sum_i W_i$ is a sum of atomic invertible potentials, with $G_W^\text{max}=\bigoplus_i G_{W_i}^\text{max}$, and similar decompositions for $W^T$.  Suppose the total unprojected mirror map prescribes
\[\prod_i H_i\ket{\oplus_i h_i}\mapsto \prod_iG_i\ket{\oplus_i g_i},\]
where $g_i\in G_{W_i}^\text{max}$ and $G_i\in \Q_{W_i|_{\Fix g_i}}$ is a monomial, and similarly for the $h_i$ and $H_i$.  For the mirror map to be an isomorphism, we require that in the $A$-model sector $\ket{\oplus g_i}$ corresponding to the $B$-model monomial $\prod_i H_i$, there is a unique monomial $\prod_i G_i$ which corresponds to $\ket{\oplus h_i}$.  This is clearly equivalent to the same holding for each atomic potential $W_i$.

Inspection of Equations \eqref{eq:bigrading} and \eqref{eq:bigrading+-} indicates that the bi-degrees are simply sums of contributions from each atomic summand, so if the total unprojected mirror map preserves bidegree for atomic potentials, it does so for all invertible potentials.

We may therefore restrict our attention to the invertible potentials of Fermat, Loop and Chain type, as the result follows for all invertible potentials from these atomic cases.  For each of these cases, we will prove that Equation \eqref{eq:Invertible Potential Total Mirror Map} is a bi-degree preserving vector space isomorphism.

By \eqref{eq:bigrading+-}, the sum of the $A$-model bi-gradings is the $A$-model degree of the $h$-twisted sector given in \cite{FJR1}.
For this reason, we will show that the bi-grading
is preserved under the Mirror Map by showing that the sum and difference of the bi-degrees is preserved.
\subsubsection{Fermat:\ $W=X^N$}
The total unprojected mirror map is defined on generators by:
\[Y^kdY\ket{\,\id\,}\longmapsto 1 \ket{\rho^{k+1}},\phantom{XXX} 0\leq k < N-1,\]
and
\[1\ket{\rho^{k+1}}\longmapsto X^k dX \ket{\,\id\,},\phantom{XXX} 0\leq k< N-1,\]
which evidently yields an isomorphism of unprojected state spaces.

To see that this isomorphism preserves bi-degree, note that for
\[\ket{\rho^{k+1}},\phantom{XXX} 0\leq k < N-1,\]
 we have
\[Q^B_+ + Q^B_- + 2\deg_\text{internal}^B = 2k \overline q,\]
and
\[Q^A_+ + Q^A_- = 2((k+1)q - q) = 2k\overline q.\]
Also, it is clear that
\[Q^B_+ - Q^B_- = 0 = Q^A_+ - Q^A_- + 2\deg_\text{internal}^A.\]

For
\[1\ket{\overline\rho^{k+1}}\longmapsto X^k dX \ket{\,\id\,},\phantom{XXX} 0\leq k< N-1,\]
we have
\[Q^B_+ + Q^B_- + 2\deg_\text{internal}^B = 1-2\overline q = Q^A_+ + Q^A_-,\]
and
\[Q^B_+ - Q^B_- = (2q-1) + 2kq = Q^A_+ - Q^A_- + 2\deg_\text{internal}^A.\]
\subsubsection{Loop: $W=\sum_{i=1}^N X_i^{a_i}X_{i+1}$}
(The subscripts are taken modulo $N$).
\newline
The structure of the loop potential means that the only group element with non-trivial fixed locus is the identity.  Therefore we study the total unprojected mirror map out of the $B$-model identity sector and twisted sectors separately.

\emph{Identity $B$-model sector:}
\[\prod_{j=1}^N Y_j^{\alpha_j}dY_j \ket{\,\id\,}\longmapsto \prod X_j^{r_j}dX_j \ket{\prod_{j=1}^N \rho_j^{\alpha_j+1}},
\]
where we are purposefully vague about the range of the product for the $A$-model monomial, since it may either be empty (in which case the monomial should be interpreted as $1$) or it may run from $1$ to $N$ (when the $B$-model monomial corresponds to the $A$-model identity group element).

In the first case, we have
\[\prod_{j=1}^N Y_j^{\alpha_j}dY_j \ket{\,\id\,}\longmapsto1 \ket{\prod_{j=1}^N \rho_j^{\alpha_j+1}},
\]
so
\begin{align*}
Q^A_+ + Q^A_-
&= \dim \Fix + 2\sum_{j=1}^N (\Theta_j-q_j)\\      &=2\sum_{j=1}^N\left(\sum_{i=1}^N(\alpha_i+1)\varphi_j^{(i)} - \sum_{i=1}^N\varphi_j^{(i)}\right)      =2\sum_{i=1}^N\alpha_i\left(\sum_{j=1}^N\varphi_j^{(i)}\right)
=2\sum_{i=1}^N\alpha_i\overline{q}_i\\
&= Q^B_+ + Q^B_- + 2\deg_\text{internal}^B
\end{align*}
and
\[
Q^B_+ - Q^B_- = 0 = Q^A_+ - Q^A_- + 2\deg_\text{internal}^A
\]
On the other hand, if $\prod_{j=1}^N \rho_j^{\alpha_j+1}=\id$, the mirror map looks like
\[\prod_{j=1}^N Y_j^{\alpha_j}dY_j \ket{\,\id=\prod_{j=1}^N \overline\rho_j^{r_j+1}}\longmapsto \prod_{j=1}^N X_j^{r_j}dX_j \ket{\prod_{j=1}^N \rho_j^{\alpha_j+1}},
\]
and we by Lemma \ref{lem:Atomic potentials -- symmetries} we must have $N$ even and $\alpha_j, r_j$ both alternately $a_j-1$ and $0$.
\begin{remark}
Evidently, there is a choice of `parity' in the mirror map.  For the purpose of establishing an isomorphism of graded vector-spaces, let us suppose that the monomials in the $X$'s and the $Y$'s have the same parity.
\end{remark}
Then,
\[Q^B_+ + Q^B_- + 2\deg_\text{internal}^B = 2\sum_{j=1}^N \alpha_j \overline q_j ,\]
and
\begin{multline*}
Q^A_+ + Q^A_- = \dim\Fix \left(\prod_{j=1}^N \rho_j^{\alpha_j+1}\right) +2\sum_{j=1}^N (\Theta_j-q_j)=\dotsc\\ \dotsc
= \sum_{j=1}^N (1-2\overline q_j)
= 2\sum_{\substack{j=1\\j\text{ odd/even}}}^N (1-\overline q_{j-1} - \overline q_j)
= 2\sum_{\substack{j=1\\j \text{ odd/even}}}^N (a_j-1) \overline q_j
= 2\sum_{j=1}^N \alpha_j \overline q_j.
\end{multline*}

Again, it is evident that
\[Q^B_+ - Q^B_- = 0,\]
while
\[Q^A_+ - Q^A_- + 2\deg_\text{internal}^A = \sum_{j=1}^N\ (2q_i-1) +2\sum _{j=1}^N r_jq_j =0,\]
by a calculation identical to the one in the preceding paragraph.

\emph{Twisted $B$-model sectors:}

Since the $B$-model twisted sectors have trivial fixed loci, the mirror map sends them all to the $A$-model untwisted sector.
\[1 \ket{\prod_{j=1}^N \overline\rho_j^{r_j+1}}\longmapsto \prod_{j=1}^N X_j^{r_j}dX_j \ket{\,\id\,},\]
Hence it is easy to see that
\[Q^B_+ + Q^B_- + 2\deg_\text{internal}^B = \hat{c} = Q^A_+ + Q^A_- .\]
Furthermore,
\[Q^B_+ - Q^B_- = 2\sum_{j=1}^N r_jq_j - \hat{c} = Q^A_+ - Q^A_- + 2\deg_\text{internal}^A\] follows after applying our computation of $\deg_+^A$ to the computation of $\deg_{-}^B$.

It is clear from the cases considered above that the total unprojected mirror map described above yields an isomorphism of vector spaces.
\subsubsection{Chain: $W=X_1^{a_1}X_2+X_2^{a_2}X_3+\dotsb +X_{N-1}^{a_{N-1}}X_N+X_N^{a_N}$}  This case is more involved than the others, because a symmetry of the chain potential may fix $\{X_s, X_{s+1}, \dotsc , X_N\}$ for any $s=1,\dotsc,N$ or it may have trivial fixed locus.

The total mirror map acts on generators via
\[\prod_{j=1}^t Y_j^{\alpha_j}dY_j \ket{\prod_{j=s}^N \overline\rho_j^{r_j+1}}\longmapsto \prod_{j=s}^N X_j^{r_j}dX_j \ket{\prod_{j=1}^t \rho_j^{\alpha_j+1}},
\]
where $\{Y_1 , \dots , Y_t\}$ are the $B$-model fixed variables and $\{X_s , \dots, X_N\}$ are the $A$-model fixed variables.  We will consider $t=0$ and $s=N+1$ to denote trivial fixed loci, and empty products and sums will be assumed to equal $1$ and $0$ respectively.

\begin{remark}
Although the group elements determine these loci,
the presentation of these elements as corresponding to specific monomials is not canonical.
\end{remark}
For the $A$-model fixed locus to be $\{X_s,\dotsc,X_N\}$, we must have $s\leq t+1$ and $\alpha_j=\delta_{t-j}^\text{even}(a_{j}-1)$ for $s\leq j\leq t$.  Further, since $Y_{s-1}\prod_{j=s}^t Y_j^{\delta_{t-j}^\text{even}(a_j-1)}$ vanishes in $\Q_{W^T|_{\{Y_1,\dots,Y_t\}}}$ when $t-s$ is even, we see that $t-s$ must be odd.  \textit{i.e.}  There is an even number of elements (possibly zero) in $\{s, s+1, \dotsc, t\}$.

This will allow us to exploit the Remark following Lemma \ref{lem:Atomic potentials -- Milnor Ring}, in which we noted that the natural map from $B$-model Milnor ring elements $\prod_i Y_i^{\alpha_i}dY_i$ to $A$-model symmetries $\prod_i \rho_i^{\alpha_i}J$ maps injectively onto the collection of symmetries with even-dimensional fixed locus.

To see that the total mirror map is a bijection for the chain potential $W$, suppose $\prod_{j=s}^N X_j^{r_j}dX_j$ is an $A$-model monomial in the target sector $\ket{g}$, where $g$ fixes the variables $\{X_s, \dotsc, X_N\}$.  To prove bijectivity, we must find in the $B$-model sector $\ket{h}=\ket{\prod_{j=s}^N \overline\rho_j^{r_j+1}}$ with fixed variables $\{1,\dotsc, t\}$ a unique monomial $\prod_{j=1}^t Y_j^{\alpha_j}dY_j$ such that $\ket{\prod_{j=1}^t \rho_j^{\alpha_j+1}} = \ket{g}$.  This can be done because for the chain $W^T|_{\{Y_1,\dotsc,Y_t\}}$, we have $\Q_{W^T|_{\{Y_1,\dotsc,Y_N\}}}$ mapping injectively onto the collection of $A$-model symmetries fixing an even number of variables $\{X_s, \dotsc,  X_t\}$.  Consequently $g$ corresponds to a unique $\prod_{i=1}^t Y_i^{\alpha_i}dY_i\in\Q_{W^T|_{\Fix h}}$ and the map is bijective as claimed.

We now proceed to compare the bi-gradings on either side of the total mirror map.
{\allowdisplaybreaks\begin{align*}
Q^A_+ + Q^A_- &=  \dim\Fix(h) +2\sum_{j=1}^N (\Theta_j-q_j)  \\
 &=  2\left(\sum_{j=1}^t\alpha_j\overline q_j \right)+2\left(\sum_{\substack{j=t+1\\N-j\text{ even}}}^N(a_j-1)\overline q_j \right)+ 2\delta_{N-(t+1)}^\text{even} \overline q_{t}-\delta_{N-s}^\text{even}\\
&= 2\left(\sum_{j=1}^t\alpha_j\overline q_j\right) + 2\left(\sum_{\substack{j=t+1\\N-j\text{ even}}}^N(1-\overline q_j - \overline q_{j-1})\right)+ 2\delta_{N-(t+1)}^\text{even} \overline q_{t}-\delta_{N-s}^\text{even}\\
&= 2\left(\sum_{j=1}^t\alpha_j\overline q_j\right)+ \left(\sum_{j=t+1}^N\left(1-2\overline q_j\right)\right)+\delta_{N-(t+1)}^\text{even}-\delta_{N-s}^\text{even}.\\
&= 2\left(\sum_{j=1}^t\alpha_j\overline q_j\right)+\left(\sum_{j=t+1}^N\left(1-2\overline q_j\right)\right)\\
&=  Q_{+}^B+Q_{-}^B+2\deg^B_\text{internal},
\end{align*}}
since we have observed that $s$ and $t+1$ have the same parity.

Now consider
{\allowdisplaybreaks\begin{align*}
Q^B_+ - Q^B_- &=  2 \sum_{j=t+1}^N(\Theta_j-\overline q_j)+ \sum_{j=t+1}^N(2\overline q_{j}-1)\\&=  t+2 \sum_{j=1}^N(\Theta_j-\overline q_j)+ \sum_{j=1}^N(2\overline q_{j}-1)\\
 &= \sum_{j=1}^t (1-2q_j)+2\left(\sum_{j=t+1}^N r_jq_j\right)+\sum_{j=1}^N(2q_{j}-1) \\
 &= \sum_{j=s}^{t-1}(1-2q_j)+2\left(\sum_{j=t+1}^N r_jq_j\right)+\sum_{j=s}^N(2q_{j}-1) \\
 &= 2\sum_{j=s}^{t-1}\delta_{t-1-j}^\text{even}(a_j-1)+2\left(\sum_{j=t+1}^N r_jq_j\right)+\sum_{j=s}^N(2q_{j}-1) \\
 &= 2\sum_{j=s}^N r_jq_j + \sum_{j=s}^N(2q_{j}-1)\\
 &=  Q^A_+ - Q^A_- + 2\deg_\text{internal}^A,
\end{align*}}
where we could change to $q$'s from $\overline q$'s because $\sum_j q_j = (1,\dotsc, 1)A^{-1}(1,\dotsc, 1)^T = \sum_j \overline q_j$.
\end{proof}
As already indicated, to complete the proof of Theorem \ref{thm:Mirror Theorem for State Spaces}, we observe simply that it is obtained by restricting the total unprojected mirror isomorphism of Theorem \ref{thm:Mirror Theorem for Total State Spaces} to $\Q_{W^T,G^T}$ to obtain an isomorphism onto its image.  By definition of the Dual Group, this image is $\h WG$.
\section{Mirror Symmetry for Frobenius Algebras}\label{sec:Mirror Symmetry for Frobenius Algebras}
\subsection{Maximal Symmetry Group}\label{sec:Maximal Symmetry Group}
We prove the following theorem,
\begin{theorem}\label{thm:Mirror Theorem for G_max}
Let $W\colon\C^N\to \C$ be a non-degenerate, invertible potential with maximal diagonal symmetry group $G^\text{max}$ and all charges $q_j < \tfrac{1}{2}$.  Let $W^T$ be the Berglund--H\"ubsch dual singularity of $W$, with Milnor ring $\Q_{W^T}$.
Then \[\h {W}{G^\text{max}}\cong \Q_{W^T}\]
as Frobenius algebra, i.e. The maximally orbifolded $A$-model ring of $W$ is isomorphic to the unorbifolded $B$-model ring of $W^T$.
\end{theorem}
The restriction to $q_j<\tfrac{1}{2}$ ensures that the ring generators of $\h{W}{G^\text{max}}$ are in Neveu-Schwartz sectors, for which the FJRW multiplication can be computed using algebro-geometric methods.

Note this corresponds to the duality of state spaces, since $G^\text{max}$ is dual to the trivial group.  However, the linear isomorphism in Theorem \ref{thm:Mirror Theorem for G_max} may in general differ from that of Theorem \ref{thm:Mirror Theorem for State Spaces}.  In the earlier theorem, there was a choice of parity involved in the presentation of the Mirror Map for loop potentials.  We have been unable to determine whether this choice is compatible with the FJRW product structure on the $A$-model.
\begin{remark}
We would like to note that in the case $N=2$, Theorem \ref{thm:Mirror Theorem for G_max} has been proven independently by Fan-Shen \cite{FanShen} in the case of chain potentials, and Acosta \cite{Acosta} in the case of loop potentials.  The Fan-Shen result applies more generally to two-variable chains with \emph{any} admissible $A$-model symmetry group.
\end{remark}
\begin{notation}
To make the notation less cumbersome in this case, we will omit the group notation for the $B$-model sector (although the reader should recall that any $B$-model monomial is implicitly followed by $dY_1\wedge\dotsb\wedge dY_N$).
\end{notation}
To prove the theorem, we recall that by combining the remark following Axiom \ref{ax:sums} and the classification of invertible potentials (\cite{KS}, recalled in Section \ref{sec:Kreuzer-Skarke Classification}), it suffices to prove Theorem \eqref{thm:Mirror Theorem for G_max} for singularities of Fermat, Loop and Chain type, which we address individually below.
\subsubsection{Fermat Potentials: $W=X^a$}
The Mirror Theorem in this case was proved as the $A_r$ case of the `self-duality' theorem in \cite{FJR1}.  The essential point here is that the exponent matrix is equal to its transpose in the self-dual cases proved in \cite{FJR1}.  We show here that the self-duality is in a sense coincidental, and that in general it is the transposed singularity ${W^T}$ which is mirror to $W$ on the $B$-model side.

\subsubsection{Loop Potentials: $W=\sum_{i=1}^N X_i^{a_i}X_{i+1}$}
Since degree is additive under multiplication in $\Q_{{W^T}}$ and in $\FJRW_{W}^{G^\text{max}}$, the isomorphism \eqref{eq:Invertible Potential Mirror Map} of graded vector spaces suggests that the desired ring isomorphism
\begin{equation}
\Q_{{W^T}}\stackrel{\cong}{\longrightarrow} \FJRW_{W,G^{\text{max}}}
\end{equation}
should be induced by the map
\begin{equation}
\begin{split}
\C[Y_1,\dotsc,Y_N]& \longrightarrow \FJRW_{W, G^{\text{max}}}\\
Y_i & \longmapsto 1_{\rho_iJ},
\end{split}
\end{equation}
where for $g\in G^{\text{max}}$, $1_g$ denotes the identity in $\FJRW_g^{G^{\text{max}}}\cong\Q_{\left.W\right|_{\Fix (g)}}^{G^{\text{max}}}$, and the map is extended to $\C[Y_1,\dotsc,Y_N]$ by multiplicativity.

The following two lemmas show that $\h{W}{G^\text{max}}$ is generated by the elements $1_{\rho_iJ}$, subject to the relations
\[(1_{\rho_kJ})^{\star a_k} + a_{k-1}1_{\rho_{k-2}J}\star (1_{\rho_{k-1}J})^{\star (a_{k-1}-1)}=0.\]
This means that the kernel of the above map is precisely the Jacobian ideal $d{W^T}$, yielding the desired isomorphism.

We proceed to prove the necessary lemmas.
\begin{notation}
        Define
\[\bs{\rho}^{\bs{\alpha}}:=\prod_{i=1}^N\rho_i^{\alpha_i}.\]
\end{notation}
\begin{lemma} \label{lem:Loop Potential multiplication polynomial relations}
  If $\alpha_i + \beta_i \leq a_i -1$ for $i\in\{1,\dotsc,N\}$, and $\id\notin\{\bs{\rho^{\alpha}}J,\bs{\rho^{\beta}}J,\bs{\rho^{\alpha+\beta}}J\}$, then
  \[1_{\bs{\rho^\alpha}J} \star 1_{\bs{\rho^\beta}J} = 1_{\bs{\rho^{(\alpha+\beta)}}J}.\]
\end{lemma}
\begin{proof}
The lemma is obviously true when $\bs{\rho^{\alpha}}= \text{id}$ or $\bs{\rho^{\beta}} = \text{id}$, since $1_J$ is the multiplicative identity in $\FJRW_W^{G^\text{max}}$.

By definition (Equation \eqref{eq:FJRW multiplication definition}),
\begin{equation*}
  1_{\bs{\rho^{\alpha}}J} \star 1_{\bs{\rho^{\beta}}J} = \sum_{\mu,\nu}\br{1_{\bs{\rho^{\alpha}}J},1_{\bs{\rho^{\beta}}J}, \mu}\eta^{\mu\nu} \nu.
\label{eq:FJRW multiplication - mirror 1}
\end{equation*}

        For the three point correlator $\br{1_{\bs{\rho^\alpha}J},1_{\bs{\rho^\beta}J}, \mu}$ to be non-zero, we must have
        \[\deg 1_{\bs{\rho^\alpha}J} + \deg 1_{\bs{\rho^\beta}J} + \deg \mu = 2\hat{c}.\]
By Corollary \ref{cor:multiplication lands in single sector}, $\mu\in\FJRW_{g J}$ for the unique $g = \bs{\rho^\gamma} \in G^\text{max}$ satisfying the condition
        \[\sum_{i=1}^N (\alpha_i+\beta_i+\gamma_i)\overline{q}_i = 2\sum_{i=1}^N (a_i -1)\overline{q}_i\]
        and having line bundles $|\mathscr L_j|$ of integral degree.

Note that since $\sum_i q_i = \sum_i\overline{q}_i$ and $a_i q_i + q_{i+1} = 1$ we have
\begin{equation*}
2\hat{c} = 2\sum_{i=1}^N (1-2q_i) = 2\sum_{i=1}^N(a_iq_i - q_i)
=2\sum_{i=1}^N (a_i-1)\overline{q}_i = \deg 1_{\bs{\rho^{\text{max}}}J}.
\end{equation*}
Since $0\leq \alpha_i + \beta _i \leq a_i - 1$ by hypothesis, $\gamma_i = a_i-1-\alpha_i-\beta_i$ potentially prescribes the group element $g$, and we demonstrate below that the corresponding line bundles indeed have integral degree.

We compute the degrees $l_j$ of the line bundles $|\mathscr L_j|$, using the formula
\[l_j = q_j(2g-2+k) - \sum_{i=1}^k\Theta^{h_i}_j,\]
Where $g$ is the genus of the correlator (zero in this case), $k$ is the number of insertions (i.e. three), $h_i\in G^\text{max}$ is the group grading of the $i^\text{th}$ insertion, and $\Theta^{h_i}_j$ is the phase of the action of $h_i$ on $X_j$.
\begin{multline*}
l_j  =  q_j - (\Theta_j^{\bs{\rho^\alpha}J} +
\Theta_j^{\bs{\rho^\beta}J} + \Theta_j^{\bs{\rho^\gamma}J})
 =  q_j - \sum_{i=1}^N (\alpha_i + 1)\varphi^{(i)}_j - \sum_{i=1}^N (\beta_i + 1)\varphi^{(i)}_j - \sum_{i=1}^N (\gamma_i + 1)\varphi^{(i)}_j\\
 =  - 2q_j - \sum_{i=1}^N (\alpha_i+\beta_i+\gamma_i)\varphi^{(i)}_j
 =  - 2q_j - \sum_{i=1}^N (a_i - 1)\varphi^{(i)}_j
 =  - 2q_j - \sum_{i=1}^N (\delta_{i,j} - \varphi^{(i-1)}_j - \varphi^{(i)}_j)
 =  -1
\end{multline*}
By the concavity axiom (Axiom \ref{ax:concavity}), $\br{1_{\bs{\rho^{\alpha}}J},1_{\bs{\rho^{\beta}}J}, \mu} = 1$.  Since $\mu$ and $\nu$ correspond to sectors with trivial fixed loci, $\eta^{\mu\nu}=1$.

We conclude on substituting into Equation \eqref{eq:FJRW multiplication - mirror 1} that
\[1_{\bs{\rho^\alpha}J} \star 1_{\bs{\rho^\beta}J} = 1_{\bs{\rho^{\alpha+\beta}}J},\]
as claimed.
\end{proof}
\begin{lemma}\label{lem:Loop Potential FJRW multiplication Jacobian Relations, N>2}
For $N>2$,
\[1_{\rho_k J}^{a_k} = - a_{k-1} 1_{\rho_{k-2}\rho_{k-1}^{a_{k-1}-1}J}.\]
\end{lemma}
\begin{proof}
By definition,
\begin{equation}
1_{\rho_k^{(a_k-1)}J}\star 1_{\rho_k J} =
\sum_{\mu,\nu}\br{1_{\rho_k^{(a_k-1)}J},1_{\rho_k J},\mu}
\eta^{\mu\nu}\nu.
\label{eq:FJRW multiplication - mirror 2}
\end{equation}

The only non-zero term in the sum occurs when
\[\deg 1_{\rho_k^{(a_k-1)}J} + \deg 1_{\rho_k J} + \deg\mu = 2\hat{c}\]
and the corresponding line bundles have integral degree.

This first condition is equivalent to
\[(a_k-1)\overline{q}_k + \overline{q}_k + \deg\mu
= \sum_{i=1}^N (a_i-1)\overline{q}_i.
\]
Recalling that for all $i$, $a_i\overline{q}_{i} + \overline{q}_{i-1}  = 1,$ so
\[\overline{q}_{k-2} + (a_{k-1}-1)\overline{q}_{k-1}  = 1 - \overline{q}_{k-1} = a_k\overline{q}_k,\]
and denoting $g=\bs{\rho^\gamma}$, we see that
\begin{equation*}
\sum_{i=1}^N \gamma_i \overline{q}_i  =  \sum_{i=1}^N (a_i-1)\overline{q}_i - (a_k-1)\overline{q}_k - \overline{q}_k
= \sum_{i=1}^N (a_i-1)\overline{q}_i - \overline{q}_{k-2} - (a_{k-1}-1)\overline{q}_{k-1}.
\end{equation*}
Thus we can solve for $\gamma_i$ in the range $0\leq\gamma_i< a_{i}$, namely
\[\gamma_i = (a_{i}-1) - \delta_{i,k-1}(a_{k-1}-1) - \delta_{i,k-2}.\]

We now confirm that the line-bundles which determine the correlator in question have integral degree, via
\begin{align*}
l_j & =  q_j - (\Theta_j^{\bs{\rho^\alpha}J} + \Theta_j^{\bs{\rho^\beta}J} + \Theta_j^{\bs{\rho^\gamma}J})\\
& =  q_j - \sum_{i=1}^N (\alpha_i + \beta_i + \gamma_i + 3)\varphi^{(i)}_j\\
& =  -2q_j - a_k\varphi^{(k)}_j - \sum_{i=1}^N \gamma_i\varphi^{(i)}_j\\
& =  -2q_j - a_k\varphi^{(k)}_j - \sum_{i=1}^N \left((a_i - 1) -\delta_{i,k-1}(a_{k-1}-1) - \delta_{i,k-2}\right)\varphi^{(i)}_j\\
& =  -2q_j - a_k\varphi^{(k)}_j - (1 - 2q_j) + (a_{k-1}-1)\varphi^{(k-1)}_j + \varphi^{(k-2)}_j\\
&=  -1 + \delta_{j,k-1} - \delta_{j,k}\\
&= \begin{cases}
0 & \text{ if } j=k-1\\
-2 & \text{ if } j=k\\
-1 & \text{else.}
\end{cases}
\end{align*}

By the index-zero axiom (Axiom \ref{ax:index-zero}), the non-vanishing three-point correlator is given by $-1$ times the $X_{k-1}$-degree of \[\frac{\partial W}{\partial X_k} = X_{k-1}^{a_{k-1}} + a_k X_k^{a_k-1}X_{k+1},\]
so
\[\br{1_{\rho_k^{(a_{k-1}-1)}J},1_{\rho_k J},\mu} = -a_{k-1}.\]

As in the preceding lemma, we have $\mu$ and $\nu$ necessarily in sectors with trivial fixed loci (i.e. not in the untwisted sector), so $\eta^{\mu\nu}=1$.

Substituting into Equation \eqref{eq:FJRW multiplication - mirror 2}, we conclude that
\[1_{\rho_k^{a_k-1}J}\star 1_{\rho_k J} = - a_{k-1} 1_{\rho_{k-2}\rho_{k-1}^{a_{k-1}-1} J},\]
as claimed.
\end{proof}
For completeness, we address the case of two-variable loop potentials in Lemma \ref{lem:Loop Potential FJRW multiplication Jacobian Relations, N=2} below.  As already indicated, this result has been obtained independently by Acosta \cite{Acosta}. \ref{thm:Mirror Theorem for G_max}.
\begin{lemma}\label{lem:Loop Potential FJRW multiplication Jacobian Relations, N=2}
For $N=2$, $k\in\{1,2\}$,
\[1_{\rho_k^{a_k} J} = - a_{k-1} 1_{\rho_{k-2}\rho_{k-1}^{a_{k-1}-1}J}.\]
\end{lemma}
\begin{proof}
The method of proof for the preceding lemma is not directly applicable here, because for a two-variable loop, $\rho_k^{a_{k-1}}J=\id$, so the multiplicands used in the proof do not all lie in Neveu-Schwartz sectors and the index-zero axiom is not directly applicable.

However, if $a_k > 3$, so $\FJRW_{\rho_k^{a_k-2} J}$ and $\FJRW_{\rho_k^2 J}$ are Neveu-Schwartz sectors, the proof of Lemma \ref{lem:Loop Potential FJRW multiplication Jacobian Relations, N>2} is easily amended to yield the same conclusion by considering the product $1_{\rho_k^{a_k} J} = 1_{\rho_k^{a_k-2} J}\star 1_{\rho_k^{2} J}$.

So it remains only to consider the cases of two-variable loop potentials with one of the exponents (which we may take to be $a_2$) equal to $2$ or $3$.

For convenience of notation, we will use variables $x:=x_1$ and $y:=x_2$, and change the subscripts in the obvious way, so for example $\rho_x := \rho_1$ and $\rho_y = \rho_2$.
\begin{itemize}
  \item  $W=x^{a_x}y + xy^3$, with $a_x\geq 3$.  To prove the lemma, we need to show $$1_{\rho_yJ}^{\star 3} = -3 1_{\rho_x^2\rho_y J}.$$
  (If $a_x=3$, then by symmetry of the $W$, the corresponding relation will hold with $x$ and $y$ exchanged.)
  Using Corollary \ref{cor:multiplication lands in single sector}, we see that
  $$1_{\rho_y J}^{\star 2} = \br{1_{\rho_yJ},1_{\rho_yJ},\mu_{\id}}\eta^{\mu,\nu}\nu_{\id},$$
  and $$1_{\rho_yJ}^{\star 3} = \br{1_{\rho_y J},1_{\rho_y J},\mu_{\id}}\eta^{\mu,\nu} \br{\nu_{\id} 1_{\rho_yJ},1_{\rho_yJ}} 1_{\rho_x^2\rho_y J}.$$
  The coefficient of $1_{\rho_x^2\rho_y J}$ is, by the composition axiom (Axiom \ref{ax:composition}) equal to
  $\br{1_{\rho_y J},1_{\rho_y J},1_{\rho_y J},1_{\rho_y J}}$.  The line bundle degrees for this correlator are
  \begin{align*}
  l_x & =  2q_x - 4\Theta_x^{\rho_y J} = 2(\tfrac{2}{8}) - 4(\tfrac{1}{8}) = 0,\\
  l_y & =  2q_y - 4\Theta_y^{\rho_y J} = 2(\tfrac{2}{8}) - 4(\tfrac{5}{8}) = -2.
  \end{align*}
  So the correlator is given by ${-1}$ times the $x$-degree of $\partial W / \partial y$. \emph{i.e}
  $\br{1_{\rho_y J},1_{\rho_y J},1_{\rho_y J},1_{\rho_y J}}={-3}$, as required.

  \item $W=x^2y + xy^3$.  The vector space generators of the
  The composition axiom argument used to compute $1_{\rho_yJ}^{\star 3}$ above applies here, yielding $1_{\rho_yJ}^{\star 3} = -2 (1_{\rho_x\rho_y J})$.

 For degree reasons, we see that $\h {W}{G^\text{max}}$ has a ring generator $\mu= \alpha x^2 dx\wedge dy + \beta y^2 dx\wedge dy\in\FJRW_{\id}$ that is not in the vector subspace generated by $1_{\rho_x J}^{\star 2} = \gamma x^2 dx\wedge dy + \delta y dx\wedge dy$.  Here $\gamma$ and $\delta$ are determined by the $\star$-product, and we seek $\alpha$ and $\beta$ so that $$\mu^2=-3 (1_{\rho_xJ}\star\mu).$$  It turns out that the matrix of the pairing $\FJRW_{\id}\otimes\FJRW_{\id}\to \C$ is given by the symmetric matrix $-\tfrac{1}{6}A_W^{-1}$.  Then, by the pairing axiom (Axiom \ref{ax:pairing}), the desired relation is equivalent to
$$
\begin{pmatrix}
\alpha & \beta
\end{pmatrix}
A^{-1}
\begin{pmatrix}
\alpha\\
\beta
\end{pmatrix}
=
-3 \begin{pmatrix}
\gamma & \delta
\end{pmatrix}
A^{-1}
\begin{pmatrix}
\alpha\\
\beta
\end{pmatrix}.
$$
Consider a non-zero vector $v$ orthogonal to $(\gamma,\delta)$ with respect to the inner product with matrix $A^{-1}$ on $\C^2$. Putting $(\alpha,\beta)=(\gamma,\delta)+\lambda v$ and substituting into the above relation, we obtain the quadratic equation
$$\lambda^2 v^T A^{-1} v + 2\begin{pmatrix}
\gamma & \delta
\end{pmatrix}
A^{-1}
\begin{pmatrix}
\gamma\\
\delta
\end{pmatrix}
=0.$$  The coefficients in this equation are non-zero, as the vanishing of either of them would contradict non-degeneracy of the form $A^{-1}$.   Either solution specifies $\mu$, which is not a multiple of $1_{\rho_xJ}^{(\star 2)}$ because the $\lambda\neq 0$.

With $\mu$ determined, the lemma follows.
  \item $W = x^2y + xy^2$.  In this case, $G^\text{max}=\br{J}$, with $J=(e^{2\pi i/3},e^{2\pi i/3})$.  The sectors $\FJRW_{J}$ and $\FJRW_{J^{-1}}$ are Neveu-Schwartz, respectively of minimal and maximal degree ($\deg_+^A$).  The identity sector is
  $\FJRW_{\id} = \C[xdx\wedge dy, ydx\wedge dy]$, and the multiplication of the generators into $\FJRW_{J^{-1}}$ is determined by the pairing axiom (Axiom \ref{ax:pairing}), from which it is easy to see that $$(xdx\wedge dy)^2 = -2 (xdx\wedge dy)(ydx\wedge dy) = (ydx\wedge dy)^2.$$
  These are precisely the defining relations for the generators of $\Q_{W^T}$, so the desired isomorphism holds.
\end{itemize}
\end{proof}
Using associativity of $A$ model multiplication to avoid the identity (Ramond) sector, it is easy to see that the mirror map is surjective.  The dimension count of Lemma \ref{lem:Atomic potentials -- Gmax invariants in twisted sector} then guarantees that the relations are generated by those in Lemma \ref{lem:Loop Potential FJRW multiplication Jacobian Relations, N>2} if $N>2$ or Lemma \ref{lem:Loop Potential FJRW multiplication Jacobian Relations, N=2} if $N=2$, from which the desired isomorphism follows.
\subsubsection{Chain Potentials: $W = \sum_{i=1}^{N-1} X_i^{a_i}X_{i+1} + X_N^{a_N}$}$\,$

Since degree is additive under multiplication in $\Q_{{W^T}}$ and in $\FJRW_{W}^{G^\text{max}}$, the isomorphism \eqref{eq:Invertible Potential Mirror Map} of graded vector spaces suggests that the desired ring isomorphism
\begin{equation}
\Q_{{W^T}}\stackrel{\cong}{\longrightarrow} \FJRW_W^{G^{\text{max}}}
\end{equation}
should be induced by the map
\begin{equation}
\begin{split}
\C[Y_1,\dotsc,Y_N]& \longrightarrow  \FJRW_{W}^{G^{\text{max}}}\\
Y_i & \longmapsto  1_{\rho_iJ},
\end{split}
\end{equation}
which is extended to $\C[Y_1,\dotsc,Y_N]$ by multiplicativity. The following two lemmas show that $\h{W}{G^\text{max}}$ is generated by the elements $1_{\rho_iJ}$, subject to the relations
\[(1_{\rho_kJ})^{\star a_k} + a_{k-1}1_{\rho_{k-2}J}\star (1_{\rho_{k-1}J})^{\star (a_{k-1}-1)}=0.\]
This means that the kernel of the mirror map is precisely the Jacobian ideal $d{W^T}$, yielding the desired isomorphism.
\begin{remark}
Note the assumption that $q_N<\tfrac{1}{2}$ is essential to our arguments, as we will use the fact that $\Fix(\rho_N J)$ is trivial.
\end{remark}
We proceed to prove the necessary lemmas.
\begin{notation}
        Define
\[\bs{\rho}^{\bs{\alpha}}:=\prod_{i=1}^N\rho_i^{\alpha_i}.\]
\end{notation}
\begin{lemma}
  If $\alpha_i + \beta_i \leq a_i -1$ for $i\in\{1,\dotsc,N\}$, and $\bs{\rho^{\alpha}}J,\bs{\rho^{\beta}}J$ and $\bs{\rho^{\alpha+\beta}}J$ have trivial fixed loci, then
  \[1_{\bs{\rho^\alpha}J} \star 1_{\bs{\rho^\beta}J} = 1_{\bs{\rho^{(\alpha+\beta)}}J}.\]
\end{lemma}
\begin{proof}
The argument here is practically identical to the one used to prove Lemma \eqref{lem:Loop Potential multiplication polynomial relations}.
\end{proof}
\begin{lemma}
\[1_{\rho_N ^{a_N-1} J}\star 1_{\rho_{N-1}J} = 0\]
\end{lemma}
\begin{proof}
Note this relation corresponds to the Jacobian relation $\frac{\partial W_\text{chain}^T}{\partial Y_{N}}=0$.

By definition (Equation \eqref{eq:FJRW multiplication definition}),
\begin{equation*}
  1_{\rho_N^{a_N-1}J} \star 1_{\rho_{N-1} J} = \sum_{\mu,\nu}\br{1_{\rho_N^{a_N-1}J}, 1_{\rho_{N-1} J}, \mu}\eta^{\mu\nu} \nu.
\end{equation*}

For the three point correlator $\br{1_{\rho_N^{a_N-1}J}, 1_{\rho_{N-1} J}, \mu}$ to be non-zero, the line bundles $|\mathscr L_j|$ must have integral degree.

We know from Corollary \ref{cor:multiplication lands in single sector} that there is at most one group element $gJ$ for which $\mu\in\FJRW_{gJ}$ yields a non-zero three point correlator.  For the sector $\FJRW_{gJ}$, let us consider the implication of integrality of the line bundles $|\mathscr L_j|$, for $j\in\{N,N-1\}$:
\begin{align*}
l_N & =  q_N - \Theta_N^{\rho_N^{a_N-1}J} - \Theta_N^{\rho_{N-1}J} - \Theta_N^{gJ}\\
& =  q_N - (a_N+1)\varphi_N^{(N)} - \Theta_N^{gJ}\\
& =  -1 - \Theta_N^{gJ}.
\end{align*}
For this to be integral, we require $\Theta_N^{gJ}\in\Z$, \emph{i.e.} $gJ$ fixes $X_N$.  Furthermore,
\begin{align*}
l_{N-1} & =  q_{N-1} - \Theta_{N-1}^{\rho_{N}^{a_{N}-1}J} - \Theta_{N-1}^{\rho_{N-1}J} - \Theta_{N-1}^{gJ}\\
& =  q_{N-1} - (a_{N}+1)\varphi_{N-1}^{(N)} - 3\varphi_{N-1}^{(N-1)} - \Theta_{N-1}^{gJ}\\
& =  -\varphi_{N-1}^{(N-1)} - \Theta_{N-1}^{gJ}.
\end{align*}
For this to be integral, we require $\Theta_{N-1}^{gJ}=1-\varphi_{N-1}^{(N-1)}\notin\Z$, \emph{i.e.} $gJ$ \emph{does not fix} $X_{N-1}$.

Since a chain potential fixes  consecutive variables, we conclude that $gJ$ has one-dimensional fixed locus, and consequently $\FJRW_{gJ}$ is empty, and the product vanishes as claimed.
\end{proof}
\begin{lemma}\label{lem:Chain Potential multiplication Jacobian Relations}
For $k\in\{2,\dotsc, N\}$,
\[1_{\rho_k J}^{\star(a_k)} = - a_{k-1} 1_{\rho_{k-2}\rho_{k-1}^{a_{k-1}-1}J}.\]
\end{lemma}
\begin{proof}
Note these relations correspond to the Jacobian relations $\frac{\partial {W^T}_\text{chain}}{\partial Y_{k-1}}=0$.

For $2 \leq k \leq N-1$, the proof proceeds exactly as in Lemma \eqref{lem:Loop Potential FJRW multiplication Jacobian Relations, N>2}.

For $k=N$, we face the obstacle that $\rho_{N}^{a_N-1}J$ is a Ramond Sector, so we cannot use the index-zero axiom as before.  We could realize $1_{\rho_NJ}^{\star(a_N)}$ as the product of $1_{\rho_NJ}^{\star 2}$ and $1_{\rho_NJ}^{\star(a_N-2)}$, but this fails to avoid the Ramond sector when $a_N=3$.  Instead, we mimic the computation in \cite{FJR1}, where the composition axiom (Axiom \ref{ax:composition}) is used to determine the ring structure of $\h{E_7}{G^\text{max}}$.

The reader may check using Corollary \ref{cor:multiplication lands in single sector} that
\[1_{\rho_N^{(a_N-2)}J}\star 1_{\rho_N J} = \br{1_{\rho_N^{(a_N-2)}J},\, 1_{\rho_N J},\, \mu 1_{\prod \rho_i^{\gamma_i} J}}\eta^{\mu,\nu}\nu 1_{\rho_N^{a_N-1}J},\]
with
\[\gamma_i=\begin{cases}
0 & \text{ if } i = N-1\\
a_i-2 & \text{ if } i = N-2\\
a_i-1 & \text{ else.}
\end{cases}.\]
Multiplying by $1_{\rho_N J}$, we see
\[1_{\rho_N J}^{\star a_N} = \br{1_{\rho_N^{(a_N-2)}J},\, 1_{\rho_N J},\, \mu 1_{\prod\rho_i^{\gamma_i} J}}\eta^{\mu,\nu} \br{\nu 1_{\rho_N^{a_N-1}J},\, 1_{\rho_NJ} ,\, 1_{\prod \rho_i^{\gamma_i-\delta_{i,N}}J}}1_{\rho_{N-2}\rho_{N-1}^{a_{N-1}}}.\]
Now, by the composition axiom,
\[\br{1_{\rho_N^{(a_N-2)}J},\, 1_{\rho_N J},\, \mu 1_{\prod\rho_i^{\gamma_i} J}}\eta^{\mu,\nu} \br{\nu 1_{\rho_N^{a_N-1}J},\, 1_{\rho_NJ} ,\, 1_{\prod \rho_i^{\gamma_i-\delta_{i,N}}J}} = \br{1_{\rho_N^{(a_N-2)}J},\, 1_{\rho_N J},\, 1_{\rho_NJ} ,\, 1_{\prod \rho_i^{\gamma_i-\delta_{i,N}}J}}.\]
Since all the sectors in this four-point correlator are Neveu-Schwartz, we may use the index-zero axiom to determine its value.  A calculation similar to the other index-zero calculations yields for the degrees of the line bundles $|\mathscr L_j|$:
\[l_j = \begin{cases}
-2 & \text{ if } j=N\\
0 & \text{ if } j= N-1\\
-1 & \text{ else.}
\end{cases}\]
So, the four-point correlator is ${-1}$ times the $X_{N-1}$ degree of $\partial W / \partial X_N = a_NX_N^{a_N-1} + X_{N-1}^{a_{N-1}}$, namely $-a_{N-1}$.  This completes the proof of Lemma \ref{lem:Chain Potential multiplication Jacobian Relations}.
\end{proof}
Surjectivity of the mirror map is again clear from associativity of $A$-model multiplication, where we avoid Ramond sectors (so we can apply the preceding lemmas) by noting that $\bs{\rho^\gamma}J$ has trivial fixed locus as long as $\gamma_N<a_{N-1}$.  A dimension count using Lemma \ref{lem:Atomic potentials -- Gmax invariants in twisted sector} then indicates that the relations in $\h{W}{G^\text{max}}$ are generated by those in the lemmas, and the desired isomorphism follows.
 \subsection{$SL$ symmetries for Calabi-Yau Loop Potentials}\label{sec:All admissible symmetries for Loop Potentials}

As evidence that the $B$-model multiplication defined in Section \ref{sec:B-model} is the appropriate product to consider in the context of LG-via-LG mirror symmetry, we prove the following theorem.
 \begin{theorem}
  Let $W(X_1,\dotsc,X_N)$ be a loop potential with $N$ odd, satisfying the Calabi-Yau condition:  $\sum_i q_i=1$.  Let $G$ be an admissible orbifold group such that $G\subset SL_N\C$.
  Then the mirror map (Equation \eqref{eq:Invertible Potential Mirror Map}) is a Frobenius algebra isomorphism.
 \end{theorem}
 \begin{remark}
 Note that the group generated by the exponential grading operator $J$ is automatically a subgroup of $SL_N\C$ in the Calabi-Yau case.

 The theorem is applicable more generally than the statement initially suggests, as the FJRW $A$-model depends only on the charges and the orbifold group, not the presentation of the singularity \cite{RuanPrivate}.  So, for example, the $J$-orbifolded $A$-models coincide for $W_\text{loop}=X_1^4X_2 +X_2^4X_3 +X_3^4X_4 +X_4^4X_5 +X_5^4X_1$ and $W_\text{Fermat}=X_1^5+X_2^5+X_3^5+X_4^5+X_5^5$, and the $J$-orbifolded $A$-model of the latter maybe computed as the $SL$-orbifolded $B$-model of the former.
 \end{remark}
 \begin{proof}
 Recall the because of the loop structure of the potential, the fixed locus for $g\in G$ is trivial unless $g=\id$.

 By Theorem \ref{thm:Mirror Theorem for State Spaces}, we know this map is a bijection.  To see that it is an isomorphism of Frobenius algebras, we consider $B$-model multiplication between untwisted sectors, between twisted sectors, and between an untwisted sector and a twisted sector.
 \emph{Untwisted $B$-model sector:}
 \[\prod_{j=1}^N Y_j^{\alpha_j}dY_j \ket{\,\id\,}\longmapsto 1 \ket{\prod_{j=1}^N  \rho_j^{\alpha_j+1}},
 \]
 where we note that since $N$ is odd, the $A$-model sector corresponding to the monomial $\prod_{j=1}^N Y_j^{\alpha_j}$ is \emph{not} the identity sector, so has trivial fixed locus.

Note that on the $A$-model side, the identity sector has degree $\hat c = N - 2\sum q_i$, which is an odd integer, while the twisted sector corresponding to a group element $g\in SL_N\C$ has degree $2\sum_i(\Theta_i^g-q_i)$, an even integer.  Since degree is additive under multiplication, the product of two Neveu-Schwartz invariants has no component in the identity sector.

Consequently, in the $A$-model product (Equation \eqref{eq:FJRW multiplication definition}), all invariants appearing with non-zero coefficient on the right-hand side are Neveu-Schwartz invariants for the action of the maximal $A$-model symmetry group, and the correlators required to determine the multiplication are as computed in the subsection on Loop potentials in  Section \ref{sec:Mirror Symmetry for Frobenius Algebras}.  \emph{i.e.} The multiplicative relations on the $A$-model twisted sectors correspond precisely to the Jacobian relations in the $B$-model untwisted sector.

 We must now consider the
 \emph{Twisted $B$-model sectors:}

Since the $B$-model twisted sectors have trivial fixed loci, the mirror map sends them all to the $A$-model untwisted sector.
\[1 \ket{\prod_{j=1}^N \overline\rho_j^{r_j+1}}\longmapsto \prod_{j=1}^N X_j^{r_j}dX_j \ket{\,\id\,},
\]
On the $A$-model side, by the pairing axiom (Axiom \ref{ax:pairing}),
\[\left(\prod_{j=1}^N X_j^{r_j}dX_j \right)1_{\id} \star
\left(\prod_{j=1}^N X_j^{s_j}dX_j \right)1_{\id}=
\br{\prod_{j=1}^N X_j^{r_j},\prod_{j=1}^N X_j^{s_j}} 1_{J^{-1}},\]
On the $B$-model side,
\[\ket{\prod_{j=1}^N \overline\rho_j^{r_j+1}} \star
\ket{\prod_{j=1}^N \overline\rho_j^{s_j+1}}\]
vanishes unless every variable is fixed in $\ket{\prod_{j=1}^N X_j^{r_j+s_j+1}dX_j}$, which means precisely that $\prod_{j=1}^N X_j^{r_j+s_j}dX_j=\lambda\hess W$ for some $\lambda\in\C$.  \emph{i.e.} the product is given by
\[\ket{\prod_{j=1}^N \overline\rho_j^{r_j+1}} \star \ket{\prod_{j=1}^N \overline\rho_j^{s_j+1}} = \tfrac{1}{\lambda}\br{\prod_{j=1}^N X_j^{r_j},\prod_{j=1}^N X_j^{s_j}}1_{\id},\]
and the products coincide up to a scalar factor for the $A$ and $B$ models.

It remains only to check that the multiplication between the twisted and untwisted $B$-model sectors satisfies the same relations as the corresponding $A$-model products.  The $B$-model $\Q_{\id}$-module structure means the only way such a product can be non-trivial is if the multiplicand from the untwisted sector is $1_{\id}$ -- the multiplicative identity.  Since the mirror map preserves the identity, we need only show that on the $A$-model side,
\[\ket{\prod_{j=1}^N \rho_j^{\alpha_j+1}}\star \prod_{j=1}^N X_j^{r_j}\ket{\id} = 0.\]
This holds for degree reasons:  the untwisted sector is the only sector with odd degree, and the twisted sectors all have even degree; by additivity of degree, the product has odd degree, so since it does not lie in the untwisted sector it must vanish.
\end{proof}
\begin{remark}
The hypotheses for this theorem ensure that there are no non-zero contributions from the Ramond sector to products of Neveu-Schwartz invariants.  The same argument will work in any case such a situation is established, so it should be possible to extend this result beyond the case of Calabi-Yau singularities orbifolded by subgroups of $SL_N\C$.
\end{remark}
\subsection{Strange Duality}\label{sec:Strange Duality}
Arnol'd's list of 14 exceptional singularities provides a source of interesting examples of Landau--Ginzburg Mirror Symmetry.  In particular, we have the following:

\begin{proposition}\label{prop:Stange Duality}
Let $W$ be one of the 14 exceptional unimodal singularities, and $W^\text{SD}$ its Strange Dual.  Then, there is  a Frobenius algebra isomorphism
\[\FJRW_W^{\br{J}}\cong \Q_{W^\text{SD}}.\]
\end{proposition}
\begin{proof}
Of course, when $J$ generates $G^\text{max}_W$ and $W^\text{SD}=W^T$, this is just a restatement of Theorem \ref{thm:Mirror Theorem for G_max}.  However, examining Table \ref{table:Strange Duality 1}, we see this is only the case for $S_{12}$, $Z_{12}$ and $E_{12}$ (which are self-dual), and $Z_{11}$ and $E_{13}$ (which are strange dual to each other).

To realize the observation for the remaining singularities in Arnol'd's list, we choose a different representative $W'$ for each singularity $W$ in such a way that
\begin{itemize}
  \item $\Q_{W'}\cong \Q_{W}$.
  \item The charges of $W'$ coincide with the charges of $W$, so $J_{W'}=J_W$.
  \item The maximal symmetry group of $W'$ is generated by $J_{W'}$.
  \item Transposition yields the Strange Dual class in the updated list of exceptional singularities.
\end{itemize}

\begin{table} \label{table:Strange Duality 1}
\caption{Arnold's list of the 14 exceptional unimodal singularities $W$, with representatives $W'$ chosen so Strange Duality is compatible with transposition.}
\begin{tabular}{|c|c|c|c|c|c|}\hline
Class & $W$ & $\br{J}=G^\text{max}_W$ & $W'$\\\hline

$Q_{10}$ & $x^2z + y^3 + z^4$ & Yes & $x^2z + y^3 + z^4$\\

$E_{14}$ & $x^2 + y^3 + z^8$ & No & $x^2 + y^3 +xz^4$\\\hline

$Q_{11}$ & $x^2z + y^3 + yz^3$ & Yes & $x^2z + y^3 + yz^3$\\

$Z_{13}$ & $x^2 + y^3z + z^6$ & No & $x^2 + y^3z + z^3x$
\\\hline

$Q_{12}$ & $x^2z + y^3 + z^5$ & No & $x^2z + y^3 +  xz^3$\\\hline

$S_{11}$ & $x^2y + y^2z + z^4$ & Yes & $x^2y + y^2z + z^4$\\

$W_{13}$ & $x^2 + y^4 + yz^4$ & No & $x^2 + xy^2 + yz^4$
\\\hline

$S_{12}$ & $x^2y + y^3z + xz^2$ & Yes & $x^2y + y^3z + xz^2$\\\hline

$U_{12}$ & $x^3 + y^3 + z^4$ & No & $x^2y + xy^2 + z^4$\\\hline

$Z_{11}$ & $x^2 + y^3z + z^5$ & Yes & $x^2 + y^3 + yz^5$\\

$E_{13}$ & $x^2 + y^3 + yz^5$ & Yes & $x^2+ y^3z + z^5$\\\hline

$Z_{12}$ & $x^2 + y^3z + yz^4$ & Yes & $x^2 + y^3z + yz^4$\\\hline

$W_{12}$ & $x^2 + y^4 + z^5$ & No & $x^2 + xy^2 + z^5$ \\\hline

$E_{12}$ & $x^2 + y^3 + z^7$ & Yes & $x^2 + y^3 + z^7$ \\\hline
\end{tabular}
\end{table}

The Landau Ginzburg $A$-model $\FJRW_{W}^G$ constructed in \cite{FJR1} is an invariant of $G\subset (\C^*)^N$ and the charges $q_1,\dotsc,\,q_N$ \cite{RuanPrivate}.  This means we are free to compute the FJRW ring of $W$ orbifolded by $\br J$ as $\FJRW_{W'}^{\br J}$.

Since $W'$ is chosen so that $J$ generates $G^\text{max}_{W'}$, we can then apply Theorem \ref{thm:Mirror Theorem for G_max} to $W'$ to yield the isomorphisms
\[\FJRW_W^{\br{J}}\cong\FJRW_{W'}^{\br{J}}\cong\Q_{(W')^T}\cong\Q_{W^\text{SD}}.\]
\end{proof}
One can attempt to use the original representative $W$ where $\br{J}$ is not necessarily $G^\text{max}$. Indeed, it fits into the general
Landau--Ginzburg Orbifold Mirror Conjecture using the orbifold $B$-model of Section \ref{sec:B-model}.

\begin{example}
We present here the case of $U_{12}$, which exhibits the general features of the other examples. We use Proposition \ref{thm:Mirror Theorem for State Spaces} to show that we have a \emph{Frobenius algebra isomorphism}
\[\FJRW_{U_{12}}^{\br{J}}\cong \Q_{U_{12}^T}^{\Z/3\Z},\]
rather than just the bi-graded vector space isomorphism guaranteed by theorem \ref{thm:Mirror Theorem for State Spaces}.

We know from the preceding discussion that
\[\FJRW_{U_{12}}^{\br{J}}\cong\Q_{U_{12}^{SD}}\cong\C[x,y,z]/\br{x^2,\,y^2,\,z^3}.\]

\begin{remark}
Note the isomorphism claimed between the Milnor rings of $U_{12}'= X^2Y + XY^2 + Z^4$ and $U_{12}=x^3+y^3+z^4$  is induced by the map
\[\C[X,Y,Z]\to \Q_{U_{12}}\]
which sends $X\mapsto \omega x + \omega^2 y$, $X\mapsto \omega^2 x + \omega y$, and $Z\mapsto z$.
\end{remark}

For the $B$-model of $U_{12}^T=x^3+y^3+z^4$, we note that the group dual to $\br{J}$ is the $SL$ subgroup of $G_{U_{12}^T}^\text{max}$, namely $\Z/3\Z$ generated by $(\omega, \omega^2, 1)$.

$\Fix (\omega,\,\omega^2,\,1)^k =
\begin{cases}
\C^3_{xyz}&\text{if } k=0\\
\C_z & \text{if } k=1,\, 2
\end{cases}$, \phantom{XX}and\phantom{XX}
$\Q_{U_{12}^T}^{\Z/3\Z}=
\begin{cases}
\br{e_0,\, ze_0,\, z^2e_0,\, xye_0,\, xyze_0,\, xyz^2e_0}&\text{if } k=0\\
e_k&\text{if } k=1,\,2,
\end{cases}$
where $e_0=dx\wedge dy\wedge dz$ and $e_1 = dz = e_2$.

We put $X=e_1$, $Y=e_2$ and $Z=ze_0$.

Note \[\deg X = \deg Y = \tfrac{1}{2}(1-2q_x)+\tfrac{1}{2}(1-2q_y)=\tfrac{1}{3},\]
while \[\deg Z = q_z=\tfrac{1}{4}.\]

We observe immediately that $Z^3=0$ (since multiplication in the untwisted sector is just multiplication in the unorbifolded Milnor ring).

Further, $X^2=0=Y^2$, since the variables $x$ and $y$ are fixed in neither the $(\omega, \omega^2, 1)$-sector, nor the $(\omega^2, \omega, 1)$-sector.

Meanwhile, $XY = \alpha xye_0$ for $\alpha\neq 0$, since $xy$ has non-zero pairing with $\hess(U_{12}|_{\C_z}) = 12z^2$.

Thus we see the degree preserving map
\[\C[X,Y,Z]\mapsto \Q_{U_{12}^T}^{\Z/3\Z}\]
defined by $X\mapsto e_1$, $Y\mapsto e_2$ and $Z\mapsto ze_0$ is surjective, with kernel $\br{Z^3,\, X^2,\, Y^2}$, and induces an isomorphism
\[\FJRW_{U_{12}}^{\br{J}}\cong\Q_{U_{12}}\mapsto \Q_{U_{12}^T}^{\Z/3\Z}.\]
\end{example}

\bibliographystyle{amsplain}
\providecommand{\bysame}{\leavevmode\hbox
to3em{\hrulefill}\thinspace}

\end{document}